\documentclass[12pt]{article}

\usepackage[normalem]{ulem}

\usepackage[margin=1in]{geometry}

\usepackage[utf8]{inputenc}

\usepackage{amsmath,amsthm}
\usepackage{amsfonts}
\usepackage{amssymb}
\usepackage{enumitem,kantlipsum}
\usepackage{makeidx}
\usepackage{tikz}
\usepackage{tikz-cd}
\usepackage{wrapfig}
\usepackage{xcolor}
\usepackage{mathtools}
\usepackage{graphicx}
\usepackage{caption}
\usepackage{comment}
\usepackage[hidelinks]{hyperref}
\usepackage{cleveref}
\usepackage{authblk}
\usepackage{titlesec}
\usepackage{listings}

\newcommand*{\red}{\textcolor{red}}

\newtheorem{theorem}{Theorem}[section]
\newtheorem{theorem*}[theorem]{*Theorem}
\newtheorem{prop}[theorem]{Proposition}
\newtheorem{lemma}[theorem]{Lemma}

\newtheorem{claim}[theorem]{Claim}
\newtheorem{conjecture}[theorem]{Conjecture}

\newcommand{\msubsection}[1]{\subsection{\texorpdfstring{#1}{e}}}
\newcommand{\msection}[1]{\section{\texorpdfstring{#1}{e}}}

\newtheorem{prop*}[theorem]{*Proposition}
\newtheorem{lema*}[theorem]{*Lemma}
\newtheorem{cor*}[theorem]{*Corollary}

\theoremstyle{definition}
\newtheorem{definition}[theorem]{Definition}
\newtheorem{question}[theorem]{Question}
\newtheorem{definition*}[theorem]{*Definition}

\newtheorem{defs*}[theorem]{*Definitions}
\newtheorem{example}[theorem]{Example}

\newtheorem{remark}[theorem]{Remark}
\newtheorem{notation}[theorem]{Notation}

\titleformat{\section}[block]{\bfseries}{\thesection}{1em}{}

\titleformat{\subsection}[block]{\bfseries\itshape}{\thesubsection}{1em}{}

\titleformat{\subsubsection}[block]{\itshape}{\thesubsubsection}{1em}{}

\titleformat{\paragraph}[runin]{\bfseries\itshape}{\theparagraph}{1em}{}[]

\titleformat{\subparagraph}[runin]{\itshape}{\thesubparagraph}{1em}{}[]

\title{\textbf{Some novel constructions of Gromov-Hausdorff-optimal correspondences between spheres}}
\author[1]{Saúl Rodríguez Martín}
\affil[1]{The Ohio State University}
\affil{\texttt{rodriguezmartin.1@osu.edu}}
\date{\vspace{-20pt}}

\begin{document}

\maketitle

\begin{abstract}
In this article, as a first contribution, we  provide alternative proofs of recent results by Harrison and Jeffs which determine the precise value of the Gromov-Hausdorff (GH) distance between the circle $\mathbb{S}^1$ and the $n$-dimensional sphere $\mathbb{S}^n$ (for any $n\in\mathbb{N}$) when  endowed with their respective geodesic metrics.
Additionally, we prove that the GH distance between $\mathbb{S}^3$ and $\mathbb{S}^4$ is equal to $\frac{1}{2}\arccos\left(\frac{-1}{4}\right)$, confirming the case $n=3$ of a conjecture by Lim, Mémoli and Smith, with computational verification of several inequalities playing an essential role in the proof.


\noindent \textbf{Word Count:} 13441 words.

\end{abstract}


\setcounter{tocdepth}{2}

\tableofcontents

\msection{Introduction}

In this article we consider the problem of determining the Gromov-Hausdorff (GH) distances between $\mathbb{S}^1$ and all other spheres, as well as  the GH distance between $\mathbb{S}^3$ and $\mathbb{S}^4$. Let us first recall some definitions. The \textit{Hausdorff distance} between two non-empty subspaces $A,B$ of a metric space $(\mathrm{X},d_{\mathrm{X}})$ is defined as
\begin{equation*}
d_{\textup{H}}^{\mathrm{X}}(A,B)=
\max\left(
\sup_{a\in A}d_{\mathrm{X}}(a,B),
\sup_{b\in B}d_{\mathrm{X}}(b,A)
\right),
\end{equation*}
where $d_{\mathrm{X}}(a,B):=\inf_{b\in B}d_{\mathrm{X}}(a,b)$. If $(\mathrm{X},d_{\mathrm{X}})$ and $(\mathrm{Y},d_{\mathrm{Y}})$ are metric spaces, we will write $\mathrm{X}\cong \mathrm{Y}$ whenever $\mathrm{X},\mathrm{Y}$ are isometric. Then the \textit{GH distance} between two nonempty metric spaces $(\mathrm{X},d_{\mathrm{X}})$ and $(\mathrm{Y},d_{\mathrm{Y}})$ is defined as 
\[
d_{\textup{GH}}(\mathrm{X},\mathrm{Y})=
\inf\{d^{\mathrm{Z}}_{\text{H}}( \mathrm{X}',\mathrm{Y}');(\mathrm{Z},d_{\mathrm{Z}})\text{ metric space; } \mathrm{X}',\mathrm{Y}'\subseteq \mathrm{Z}; \mathrm{X}'\cong  \mathrm{X};\mathrm{Y}'\cong \mathrm{Y}\}.
\]
The GH distance takes values in $[0,\infty]$, and it satisfies the triangle inequality (cf. \cite{BBI} Prop. 7.3.16). If $\mathrm{X}$ and $\mathrm{Y}$ are compact metric spaces, then $d_{\mathrm{GH}}(\mathrm{X},\mathrm{Y})=0$ iff $\mathrm{X}$ and $\mathrm{Y}$ are isometric, and the value $d_{\mathrm{GH}}(\mathrm{X},\mathrm{Y})$ is sometimes described as a way to measure how far the spaces $\mathrm{X}$ and $\mathrm{Y}$ are from being isometric.

\begin{example}\label{densedGH=0}
If $\mathrm{X}$ is a dense subspace of $\mathrm{Y}$, then $d_{\text{GH}}(\mathrm{X},\mathrm{Y})=0$.
\end{example}

Since it was introduced by Edwards \cite{Ed} and independently by Gromov \cite{Gr}, the GH distance has been instrumental in research areas such as the analysis of shapes formed by point cloud data \cite{MS,BBK}, convergence results for sequences of Riemannian manifolds \cite{CC,PW,Co1,Co2}, differentiability in metric measure spaces \cite{Ke,Ch}, and the robustness of topological invariants of metric spaces when they suffer small deformations \cite{Pe,Ro,CCG,CSO}.

Recently, there has been growing interest (see e.g. \cite{LMS,Polymath,HJ}, and for a historical account of these efforts see \cite{MS}) in computing the exact value of the GH distance between certain simple metric spaces, specifically round spheres $\mathbb{S}^n\subseteq\mathbb{R}^{n+1}$ ($n\in\mathbb{N}:=\{1,2,\dots\}$) equipped with the geodesic metric $d_{\mathbb{S}^n}$. 
In their paper \cite{LMS}, Lim, Mémoli and Smith provided some upper and lower bounds for $d_{\mathrm{GH}}(\mathbb{S}^n,\mathbb{S}^m)$ for all $n,m\in\mathbb{N}$ and they gave exact values for the pairwise distances between $\mathbb{S}^1,\mathbb{S}^2$ and $\mathbb{S}^3$. 
Some bounds for the GH distances between spheres were further improved in \cite{Polymath}. Most importantly for our purposes, a concrete case of \cite[Theorem B]{LMS} implies that  \begin{equation}\label{trefdsterfd}
2d_{\mathrm{GH}}(\mathbb{S}^n,\mathbb{S}^{n+1})\geq\zeta_n:=\arccos\left(\frac{-1}{n+1}\right),
\textup{for all }n\in\mathbb{N},
\end{equation}
leading to the following conjecture.
\begin{conjecture}[{\cite[Conjecture 1]{LMS}}]

    \label{ConjdSnSn+1}

    For all $n\in\mathbb{N}$ we have $d_{\textup{GH}}\left(\mathbb{S}^n,\mathbb{S}^{n+1}\right)=\frac{1}{2}\zeta_n$.

\end{conjecture}

As a consequence of the improved bounds provided in \cite{Polymath} the authors obtain in Theorem 5.1 the following 
\begin{equation}\label{trefdsterfd2}
d_{\text{\normalfont GH}}(\mathbb{S}^1,\mathbb{S}^{2n})\geq\frac{\pi n}{2n+1}\text{ and }
d_{\text{\normalfont GH}}(\mathbb{S}^1,\mathbb{S}^{2n+1})\geq\frac{\pi n}{2n+1}
\textup{ for all }n\in\mathbb{N}.
\end{equation}

Harrison and Jeffs proved in \cite{HJ} that the inequalities from \Cref{trefdsterfd2} are actually equalities: 

\begin{theorem}[{\cite[Theorem 1.1]{HJ}}]\label{S1toSeven}
For any integer $n\geq1$, $d_{\text{\normalfont GH}}(\mathbb{S}^1,\mathbb{S}^{2n})=\frac{\pi n}{2n+1}$.
\end{theorem}

\begin{theorem}[{\cite[Theorem 5.3]{HJ}}]\label{S1toSodd}
For any integer $n\geq1$, $d_{\text{\normalfont GH}}(\mathbb{S}^1,\mathbb{S}^{2n+1})=\frac{\pi n}{2n+1}$.
\end{theorem}
The main objectives of this article are the following.
\begin{enumerate}
    \item We give an alternative proof of \Cref{S1toSeven}. The construction we use is an immediate generalization of the one in \cite[Appendix D]{LMS}. Our proof was found independently of \cite{HJ}, but it is similar to it, as we explain in more detail in \Cref{RmkEvenMapHJ}.
    \item We give a proof of \Cref{S1toSodd} which is distinct from (and considerably shorter than) the one in \cite{HJ}. The proof in \cite[Section 3]{HJ} uses an `embedding-projection correspondence' (see \cite[Section 3]{MS}) to prove \Cref{S1toSodd}, while we instead use a certain modification of the alternative  construction we developed for proving \Cref{S1toSeven}.

    \item We also establish the following novel result, proving case $n=3$ of \Cref{ConjdSnSn+1}.
    
\end{enumerate}

\begin{theorem}
    \label{S3toS4}
    The distance 
    $d_{\text{\normalfont GH}}(\mathbb{S}^3,\mathbb{S}^4)$ is $\frac{1}{2}\zeta_3$.
\end{theorem}

    The proof strategy  of \Cref{S3toS4} is also valid for cases $n=1,2$ of \Cref{ConjdSnSn+1}, and perhaps it could be adapted to cases $n=4,5,6$ (see \Cref{Rmkn=456}).

The definition of $d_{\text{GH}}$ is hard to work with; we now recall (cf. \S7.3 in \cite{BBI}) an equivalent definition based on correspondences between sets. Recall that a \textit{relation} between two sets $\mathrm{X},\mathrm{Y}$ is a subset of $\mathrm{X}\times \mathrm{Y}$. We will say a relation $R\subseteq \mathrm{X}\times \mathrm{Y}$ is a \textit{correspondence} between $\mathrm{X}$ and $\mathrm{Y}$ if $\pi_{\mathrm{X}}(R)=\mathrm{X}$ and $\pi_{\mathrm{Y}}(R)=\mathrm{Y}$, where $\pi_{\mathrm{X}}:\mathrm{X}\times \mathrm{Y}\to \mathrm{X}$ and $\pi_{\mathrm{Y}}:\mathrm{X}\times \mathrm{Y}\to \mathrm{Y}$ are the coordinate projections.

If $(\mathrm{X},d_{\mathrm{X}})$ and $(\mathrm{Y},d_{\mathrm{Y}})$ are metric spaces, we will define the distortion of a nonempty relation $R\subseteq \mathrm{X}\times \mathrm{Y}$ as 
\begin{equation}
\text{dis}(R):=\sup\{|d_{\mathrm{X}}(x,x')-d_{\mathrm{Y}}(y,y')|;(x,y),(x',y')\in R\}\in[0,\infty].
\end{equation}
In Theorem 7.3.25 of \cite{BBI} it is proved that, if $(\mathrm{X},d_{\mathrm{X}}),(\mathrm{Y},d_{\mathrm{Y}})$ are metric spaces, 
\begin{equation}\label{distortandGH}
d_{\text{GH}}(\mathrm{X},\mathrm{Y})=\frac{1}{2}\inf\{\text{dis}(R);R\subseteq \mathrm{X}\times \mathrm{Y}\text{ correspondence between $\mathrm{X}$ and }\mathrm{Y}\}.
\end{equation}

So the distortion of any correspondence $R$ between $\mathrm{X}$ and $\mathrm{Y}$ is an upper bound for  $2d_{\text{GH}}(\mathrm{X},\mathrm{Y})$.

\begin{example}
The graph of a function $\phi:\mathrm{X}\to \mathrm{Y}$, $G_\phi$, is a relation, and it will be a correspondence between $\mathrm{X}$ and $\mathrm{Y}$ iff $\phi$ is surjective. We then define the distortion of $\phi$ as

$$\textup{dis}(\phi):=\textup{dis}(G_\phi)=\sup\{|d_{\mathrm{X}}(x,x')-d_{\mathrm{Y}}(\phi(x),\phi(x'))|;x,x'\in \mathrm{X}\}.$$
\end{example}

\begin{remark}\label{trhgfdtrgf}
We can slightly relax the definition of correspondence: we say $R\subseteq \mathrm{X}\times \mathrm{Y}$ is a \textit{metric correspondence} between $(\mathrm{X},d_{\mathrm{X}})$ and $(\mathrm{Y},d_{\mathrm{Y}})$ if the projections $\pi_{\mathrm{X}}(R),\pi_{\mathrm{Y}}(R)$ are dense in $\mathrm{X}$, $\mathrm{Y}$ respectively. If $R$ is a metric correspondence between $\mathrm{X}$ and $\mathrm{Y}$, then the triangle inequality for $d_{\text{GH}}$ and \Cref{densedGH=0} imply that $\frac{1}{2}\text{dis}(R)$ is an upper bound for $d_{\text{GH}}(\mathrm{X},\mathrm{Y})$.
\end{remark}

Thanks to \Cref{distortandGH}, in order to prove \Cref{S1toSeven,S1toSodd,S3toS4} it suffices to construct metric correspondences between spheres having adequate distortions. 

All of the correspondences we construct are related to constructions in \cite{LMS}, and they have several points in common.  Firstly, all of them use the helmet trick (\Cref{HelmetTrick} below), which tells us that, if $n,m\in\mathbb{N}$ and $H^{n}_+:=\{x\in\mathbb{S}^{n};x_{n+1}\geq0\}$, then for any correspondence $R\subseteq H^{n}_+\times\mathbb{S}^m$ there is a correspondence $R'\subseteq\mathbb{S}^n\times\mathbb{S}^m$ containing $R$ and satisfying $\textup{dis}(R')=\textup{dis}(R)$. When estimating the distance $d_{\textup{GH}}(\mathbb{S}^n,\mathbb{S}^m)$, this allows us to use correspondences in $H^{n}_+\times\mathbb{S}^m$, which are easier to construct than correspondences in $\mathbb{S}^n\times\mathbb{S}^m$.

Secondly, our constructions all use regular simplices inscribed in $\mathbb{S}^n$, by which we mean a set of $n+2$ distinct points $p_1,\dots,p_{n+2}\in\mathbb{S}^n$ such that $d_{\mathbb{S}^n}(p_i,p_j)$ is the same for all $i,j$ with $i\neq j$. A list of useful properties of such simplices can be found in \Cref{Viprops}.

\subsubsection*{Structure of the paper}
In \Cref{SectPrelims} we introduce notation and preliminaries necessary for the rest of the article.

In \Cref{SectionS1Seven} we prove \Cref{S1toSeven}, which has the easiest proof of our three main results. To do this, we devise a metric correspondence $R_{2n}\subseteq H^{2n}_+\times\mathbb{S}^1$ with distortion $\frac{2\pi n}{2n+1}$. The correspondence $R_{2n}$ is an immediate generalization of a construction from \cite[Appendix D]{LMS}.

In \Cref{SecS1Sodd} we prove \Cref{S1toSodd}. 
To explain the approach we take in that section: given \Cref{trefdsterfd2} and the fact that $d_{\textup{GH}}(\mathbb{S}^1,\mathbb{S}^{2n})=\frac{\pi n}{2n+1}$ (proved in 
\Cref{S1toSeven}), one could optimistically conjecture that $d_{\textup{GH}}(\mathbb{S}^1,\mathbb{S}^{2n+1})=d_{\textup{GH}}(\mathbb{S}^1,\mathbb{S}^{2n})=\frac{\pi n}{2n+1}$. Therefore, a natural approach is using the correspondence from \Cref{SectionS1Seven} to create some correspondence $R_{2n+1}\subseteq H^{2n+1}_+\times\mathbb{S}^1$. And indeed, we start with a natural adaptation of the correspondence $R_{2n}$ to dimension $2n+1$ and after `rotating' it in a small subset $B\subseteq H^{2n+1}_+$ (shown in \Cref{FigureAandB}), we obtain a correspondence $R_{2n+1}\subseteq H^{2n+1}_+\times\mathbb{S}^1$ with distortion $\frac{2\pi n}{2n+1}$. The way in which we rotate the correspondence in the set $B\subseteq H^{2n+1}_+$ is inspired by the arguments from \cite[Section 7]{LMS}.

In \Cref{SecDS3S4} we prove \Cref{S3toS4} using a surjective map $F:H^{4}_+\to\mathbb{S}^3$ with distortion $\zeta_3=\arccos\left(\frac{-1}{4}\right)$. The construction of the map $F$ is not particularly complicated, but the author has only found proofs that $\textup{dis}(F)=\zeta_3$ using computer assistance, see \Cref{CompAssistIneq}. \Cref{SecDS3S4} can be read without first reading Sections \ref{SectionS1Seven} and \ref{SecS1Sodd}, but requires some spherical geometry results which we include separately in \Cref{SecAppendix}.

\subsubsection*{Main idea of \Cref{SecDS3S4}}
Of the three theorems presented above, the proof of \Cref{S3toS4} is the most intricate and partly relies on computer assistance. In what follows, we outline the main ideas behind the construction of the map $F:H^{4}_+\to\mathbb{S}^3$ with distortion $\zeta_3=\arccos\left(\frac{-1}{4}\right)$, and describe how computer verification enters the argument.

Roughly speaking, the map $F$ is obtained as an `interpolation' between two functions $F',F'':H^4_+\to\mathbb{S}^3$ described below and depicted in \Cref{Mapsinter37}.

We obtain $F':H^4_+\to\mathbb{S}^3$ by taking points $p_1,\dots,p_{5}$ forming a regular simplex inscribed in $\mathbb{S}^3\equiv\{x\in\mathbb{S}^4\subseteq\mathbb{R}^5;x_{5}=0\}$ and for each $x\in H^4_+$ defining $F'(x)=p_i$, where $i\in\{1,\dots,5\}$ is chosen so that $p_i$ is as close as possible to $x$. The map $F'':H^4_+\to\mathbb{S}^3$ is obtained by, for each $p\in H^4_+$, choosing $F''(p)$ to be the point of $\mathbb{S}^3$ which minimizes the distance to $p$. So $F''(p)$ is the `projection' of $p$ to $\mathbb{S}^3$, except if $p$ is the north pole $N:=(0,0,0,0,1)$.

\begin{figure}[h]
    \centering
    \includegraphics[width=0.4\textwidth]{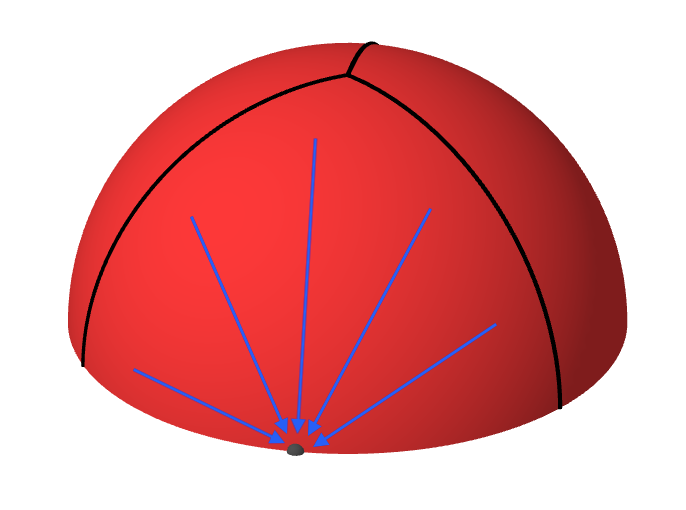}
    \hfill\includegraphics[width=0.4\textwidth]{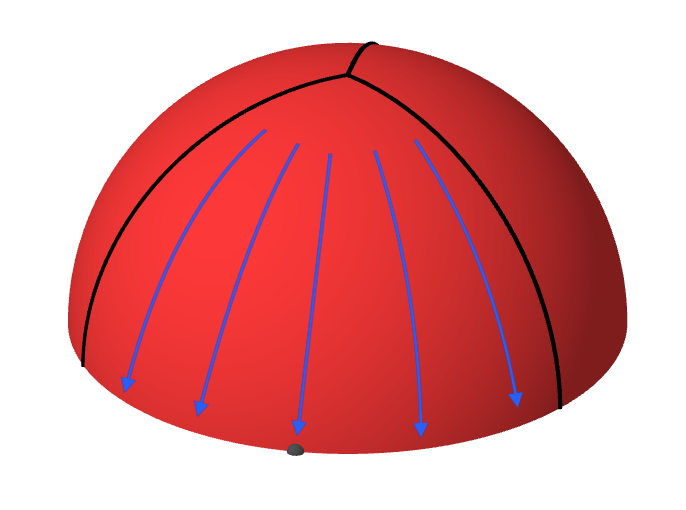}
    \caption{A low dimensional depiction of the maps $F'$ (left) and $F''$ (right).}
    \label{Mapsinter37}
\end{figure}
We want our function $F$ to have distortion $\zeta_3$. The map $F'$ has distortion $\zeta_3$, but it does not induce a correspondence. The map $F''$ has the opposite problem: it is surjective but has distortion $\pi$, because for points $x,x'$ very close to the north pole,  $d_{\mathbb{S}^3}(F(x),F(x'))-d_{\mathbb{S}^4}(x,x')$ may be as close to $\pi$ as we want. In the proof of \cite[Proposition 1.16]{LMS}, they find $d_{\mathrm{GH}}(\mathbb{S}^1,\mathbb{S}^2)$ via a surjective map $\phi:H^2_+\to\mathbb{S}^1$ which is equal to (the lower dimensional analogue of) $F'$ for points in the equator and equal to $F''$ in the rest of $H^2_+$. 
The analogous map in higher dimension, $\phi_n:\mathbb{S}^{n+1}\to\mathbb{S}^n$, has distortion $\eta_n>\zeta_n$ (defined in \Cref{Viprops}) for $n\geq3$, so it cannot be used to prove $d_{\mathrm{GH}}(\mathbb{S}^{n+1},\mathbb{S}^n)=\frac{1}{2}\zeta_n$. 

Our map $F$ is defined as $F'$ for points near the north pole (so that its distortion is not $\pi$), $F''$ for points in $\mathbb{S}^3$ (so that $F$ is surjective) and an interpolation between $F'$ and $F''$ in between. A depiction of the map $F$ can be found in \Cref{ExampleF}.
In \Cref{SecDS3S4} we define a higher dimensional analogue of $F$, which we call $F_n:\mathbb{S}^{n+1}\to\mathbb{S}^n$.
We have $\mathrm{dis}(F_n)>\zeta_n$ when $n\geq7$ (see \Cref{SecDifIneq}), so in that case it cannot be used to prove \Cref{ConjdSnSn+1}.

\begin{question}\label{34erwfdscpo}
Do we have $\textup{dis}(F_n)=\zeta_n$ for $n=4,5,6$?
\end{question}

In the rest of the introduction we explain how we use computer assistance to prove some inequalities needed in \Cref{SecDS3S4}, as well as give some ideas that could be used to prove \Cref{34erwfdscpo}.

\begin{remark}[Computer assisted proofs of inequalities]\label{CompAssistIneq}
In \Cref{SecDS3S4}, when we prove that the map $F:\mathbb{S}^4\to\mathbb{S}^3$ has distortion $\zeta_3$, we need to prove inequalities of the form $f(x_1,x_2)\geq 0$, where $x_i$ are in intervals $[a_i,b_i]\subseteq\mathbb{R}$. The expression for the function $f$ is sometimes very complicated, and it is not clear how to give a clean proof of the inequality. However, we can prove $f(x_1,x_2)\geq0$ using brute force if two conditions are met:
\begin{enumerate}
    \item\label{CompIneq1} The function $f$ is bounded below by some constant $\varepsilon>0$. So, if we suspect $f(x,y)$ tends to $0$ at some point in the boundary of the domain, then we would need to take a slightly smaller domain and use a different method to prove the inequality near the boundary.
    \item \label{CompIneq2}The function $f$ is uniformly continuous: there is a constant $\delta$ (which we can compute explicitly) such that, if $|x_1-x_1'|,|x_2-x_2'|<\delta$, then $|f(x_1,x_2)-f(x_1',x_2')|<\varepsilon$.
\end{enumerate}
Under these conditions we can use a computer program to check that $f(x_1,x_2)>\varepsilon$ for all points $(x_1,x_2)$ in some finite set $G$ (a grid) inside $[a_1,b_1]\times[a_2,b_2]$ such that every point of $[a_1,b_1]\times[a_2,b_2]$ is at distance $<\delta$ of some point of $G$, concluding the inequality.

We have used this method to check three crucial inequalities in \Cref{SecDifIneq}. The python code can be found in the GitHub repository \cite{Git}; both files that check the inequality and files that output a 3D plot of the function $f(x,y)$ are included. Each time we use this reasoning to prove an inequality, we have explained in detail how to obtain the uniform continuity constants and included a graph of the function $f(x_1,x_2)$. 

Our computations implicitly assume that the programs, implemented in Python/NumPy, evaluate $f(x,y)$ with an absolute error smaller than $10^{-2}$ for each fixed value of $(x,y)$. To justify this assumption, note that Python employs double-precision floating-point arithmetic, with machine precision on the order of $10^{-15}$. Furthermore, the expressions defining our functions consist primarily of $1$-Lipschitz components, like $\sin$ or $\cos$ (which add error at most $\sim10^{-15}$), except for at most two occurrences per expression where we evaluate $\arccos(x)$ near $x=\pm1$, or $\sqrt{x}$ near $x=0$. These cases can convert an error of $\varepsilon$ to an error of at most $\sqrt{\varepsilon}$, which still keeps the error far below $10^{-2}$.
\end{remark}

\begin{remark} \label{Rmkn=456}
If the answer to \Cref{34erwfdscpo} is positive, it should be theoretically  provable using the same ideas of \Cref{CompAssistIneq}\footnote{One first needs to use mathematical arguments to ensure that conditions \ref{CompIneq1} and \ref{CompIneq2} of \Cref{CompAssistIneq} are satisfied, as we do in \Cref{SecDS3S4}}. Indeed, \Cref{34erwfdscpo} reduces to the inequality
\begin{equation*}
|d_{\mathbb{S}^{n+1}}(x,x')-d_{\mathbb{S}^n}(F_n(x),F_n(x'))|\leq\zeta_n\textup{ for all }x,x'\in\mathbb{S}^{n+1}\text{, for }n=4,5,6.
\end{equation*}
However, the author has not been able to find a computer program efficient enough to prove dist$(F_n)=\zeta_n$ using this strategy in a reasonable amount of time, the main obstacle being that grids in $\mathbb{S}^n$ have too many points for $n\geq4$.
\end{remark}

\paragraph{Acknowledgements.} Special thanks to Facundo Mémoli for introducing the author to the study of GH distances and providing guidance and extensive feedback
while writing this article. I appreciate Sunhyuk Lim and the anonymous referees, who carefully read the article, found many typos and provided useful suggestions. Daniel Hurtado, River Li and Pablo Vitoria also helped improve some of the arguments in this article. The large language models ChatGPT-4 and Claude 3.5-Sonnet were used to aid in writing part of the python code.

The author gratefully acknowledges support from the
grants BSF 2020124 and NSF CCF AF 2310412. The author reports that there are no competing interests to declare.

\msection{Notation and preliminaries}\label{SectPrelims}

Throughout most of this article, the metric spaces we will be studying are the unit spheres $\mathbb{S}^n$ with the geodesic distance, which will be denoted by $d_{\mathbb{S}^n}$. That is, considering $\mathbb{S}^n:=\{x\in\mathbb{R}^{n+1};\sum_{i=1}^{n+1}x_i^2=1\}$, for all $x,y\in\mathbb{S}^n$ we will have 
\[d_{\mathbb{S}^n}(x,y):=\arccos\left(\langle x,y\rangle\right)\text{, where }\langle x,y\rangle:=\sum_{i=1}^{n+1}x_iy_i.\] Thus distances take values in the interval $[0,\pi]$. For any two non-antipodal points $x,x'\in\mathbb{S}^n$ we will denote the (unique) geodesic segment from $x$ to $x'$ by $[x,x']$.  Also, if $p,q,r\in\mathbb{S}^{n+1}$ are distinct points, then we denote by $\angle pqr\in[0,\pi]$ the angle at $q$ of the spherical triangle with vertices $p,q,r$, as specified by the spherical cosine rule:
\begin{equation*}
\cos\left(d_{\mathbb{S}^n}(p,r)\right)
=
\cos\left(d_{\mathbb{S}^n}(p,q)\right)
\cos\left(d_{\mathbb{S}^n}(r,q)\right)
+
\sin\left(d_{\mathbb{S}^n}(p,q)\right)
\sin\left(d_{\mathbb{S}^n}(r,q)\right)
\cos(\angle pqr).
\end{equation*}

\begin{notation}[Antipodal sets]
For any subset $\mathrm{X}\subseteq\mathbb{S}^n\subseteq\mathbb{R}^{n+1}$, we define 
\[
-\mathrm{X}:=\{-x;x\in \mathrm{X}\}.
\]
Similarly, for any relation $R\subseteq\mathbb{S}^n\times\mathbb{S}^m$, we define
\[
-R:=\{(-x,-y);(x,y)\in R\}.
\]
\end{notation}

The following is a version of Lemma 5.5 of \cite{LMS} for relations:
\begin{prop}[Helmet trick for relations, {cf. \cite[Lemma 5.5.]{LMS}}]\label{HelmetTrick}
Let $R\subseteq \mathbb{S}^n\times\mathbb{S}^m$ be a relation and let $-R=\{(-x,-y);(x,y)\in R\}$. Then the relation 
\[-R\cup R\subseteq\mathbb{S}^n\times\mathbb{S}^m\] has the same distortion as $R$.
\end{prop}

\begin{proof}
For any $k$ and for any two points $x,y\in\mathbb{S}^k$, we have $d_{\mathbb{S}^k}(x,y)=\pi-d_{\mathbb{S}^k}(x,-y)$. So for any two pairs $(x,y),(x',y')\in\mathbb{S}^n\times\mathbb{S}^m$, 
\[
|d_{\mathbb{S}^n}(x,-x')-d_{\mathbb{S}^m}(y,-y')|=|(\pi-d_{\mathbb{S}^n}(x,x'))-(\pi-d_{\mathbb{S}^m}(y,y'))|=|d_{\mathbb{S}^n}(x,x')-d_{\mathbb{S}^m}(y,y')|.
\]
Thus, any distortion $|d_{\mathbb{S}^n}(x,x')-d_{\mathbb{S}^m}(y,y')|$ between two pairs of points $(x,y)$ and $(x',y')$ of $-R\cup R$ is also attained between two pairs of points of $R$, which proves $\text{dis}(-R\cup R)=\text{dis}(R)$.
\end{proof}

\begin{definition}
Fix a point $q\in\mathbb{S}^{n+1}$. For any nonempty $\mathrm{X}\subseteq\{x\in\mathbb{S}^{n+1};\langle x,q\rangle=0\}$, we define the cone $C_q\mathrm{X}$ as the union of geodesic segments 
\begin{equation}\label{ConeOperation}
C_q\mathrm{X}:=\bigcup_{x\in\mathrm{X}}[x,q].
\end{equation}
\end{definition}

As proved in Lemma 6.1 of \cite{LMS}, the diameter of a cone $C_q\mathrm{X}$ is given by
\begin{equation}\label{DiamCones}
\text{diam}(C_q\mathrm{X})=\max\left(\frac{\pi}{2},\text{diam}(\mathrm{X})\right).
\end{equation}

For each $n\in\mathbb{N}$, we can find distinct points $p_1,\dots,p_{n+2}\in\mathbb{S}^n$ which form a regular simplex in $\mathbb{R}^{n+1}$. We can associate to them the open Voronoi cells
\begin{equation}
\label{DefVoronoi}
V_i:=\{x\in\mathbb{S}^n;d_{\mathbb{S}^n}(x,p_i)<d_{\mathbb{S}^n}(x,p_j)\text{ for all }j\neq i\}, i=1,2,\dots,n+2.
\end{equation}
We will say a subset $A\subseteq\mathbb{S}^n$ is convex when any geodesic segment between two points of $A$ is contained in $A$ (so an open hemisphere is convex, and a closed hemisphere is not). The convex hull of a set $A\subseteq\mathbb{S}^n$ is the intersection of all convex sets containing $A$.

We will need some properties of regular simplices inscribed in $\mathbb{S}^n$ (some of them are proved in \cite{LMS} and \cite{Sa};  see also Section 3 of \cite{Cho} for related results):
\begin{prop}[Properties of regular simplices in $\mathbb{S}^n$]\label{Viprops} Let $(p_i)_{i=1}^{n+2}$ and $(V_i)_{i=1}^{n+2}$ be as above. Then
\begin{enumerate}[label=\alph*)]
    \item \label{Viprops35435}$d_{\mathbb{S}^n}(p_i,p_j)=\zeta_n:=\arccos\left(\frac{-1}{n+1}\right)$ for $i\neq j$. Thus, $\frac{2\pi}{3}=\zeta_1>\zeta_2>\zeta_3>\cdots\stackrel{n\to\infty}\longrightarrow\frac{\pi}{2}$.
    
    \item \label{Viprops52345} For all $i$, the closure $\overline{V_i}$ of $V_i$ is the convex hull inside $\mathbb{S}^n$ of the set $\{-p_j;j\neq i\}$.
    
    \item \label{Viprops545543}The diameter of the Voronoi cells $V_i$ is 
    \[
    \eta_n:=\left\{
    \begin{array}{cc}
    \arccos\left(-\frac{n+1}{n+3}\right)&\text{for $n$ odd}\\
    \arccos\left(-\sqrt{\frac{n}{n+4}}\right)&\text{for $n$ even.}
    \end{array}\right.
    \]
    
    \item \label{Viprops986865}The Voronoi cell $V_i$ satisfies
    \[
    B_{\mathbb{S}^n}\left(p_i,\frac{\zeta_n}{2}\right)\subseteq V_i\subseteq B_{\mathbb{S}^{n}}\left(p_i,\pi-\zeta_n\right),
    \]
    where $B_{\mathbb{S}^n}(x,r)$ denotes the open ball centered at $x$ of radius $r$ in $\mathbb{S}^n$.
\end{enumerate}
\end{prop}

\begin{proof}
\Cref{Viprops52345,Viprops545543} are discussed in Remarks 6.4 and 6.5 of \cite{LMS}. To prove \Cref{Viprops35435} note that, as $p_1,\dots,p_{n+2}$ are unit vectors forming a regular simplex centered at $0$, we have $\sum_{i=1}^{n+2}p_i=0$. Moreover, by symmetry the scalar products $\langle p_i,p_j\rangle$ take a single value $\kappa$ for any pair $(i,j)$ with $i\neq j$. This value $\kappa$ can be obtained from the equation
\[
0=\left\langle p_i,0\right\rangle=
\left\langle p_i,\sum_{j=1}^{n+2}p_j\right\rangle
=
1+(n+1)\kappa.
\]
So $\cos(d_{\mathbb{S}^{n}}(p_i,p_j))=\frac{-1}{n+1}$ for all $i\neq j$, as we wanted.

Finally we prove \Cref{Viprops986865}. The first containment is a consequence of the fact that, if $d_{\mathbb{S}^{n}}(x,p_i)<\frac{\zeta_n}{2}$, then for any $j\neq i$ we have
\[
d_{\mathbb{S}^{n}}(x,p_j)\geq 
d_{\mathbb{S}^{n}}(p_i,p_j)-d_{\mathbb{S}^{n}}(x,p_i)
>\zeta_n-\frac{\zeta_n}{2}>d_{\mathbb{S}^{n}}(x,p_i),
\]
so $x\in V_i$. For the second containment, note that $\overline{B_{\mathbb{S}^{n}}\left(p_i,\pi-\zeta_n\right)}$ is convex (balls of radius $<\frac{\pi}{2}$ are convex in $\mathbb{S}^n$) and contains the points $-p_j$ for all $j\neq i$. So by \Cref{Viprops52345} we have $\overline{V_i}\subseteq\overline{B_{\mathbb{S}^{n}}\left(p_i,\pi-\zeta_n\right)}$, and as $V_i$ is open we also have $V_i\subseteq B_{\mathbb{S}^{n}}\left(p_i,\pi-\zeta_n\right)$.
\end{proof}

\msection{Distance from $\mathbb{S}^1$ to even dimensional spheres}\label{SectionS1Seven}
In this section, we present a new proof of \Cref{S1toSeven} by constructing a correspondence between $\mathbb{S}^1$ and $\mathbb{S}^{2n}$ with distortion $\frac{2\pi n}{2n+1}$. This construction generalizes the one used in Appendix $D$ of \cite{LMS} to find $d_{\text{GH}}(\mathbb{S}^1,\mathbb{S}^2)$.

Let us start with some notation:
\[
\mathbb{S}^{2n}:=\left\{x\in\mathbb{R}^{2n+1};\sum_{i=1}^{2n+1}x_i^2=1\right\}
\]
\[
H^{2n}_+:=\{x\in\mathbb{S}^{2n};x_{2n+1}\geq0\}
\]
\[
\mathbb{S}^{2n-1}:=\{x\in\mathbb{S}^{2n};x_{2n+1}=0\}
\]

Let $p_1,\dots,p_{2n+1}\in\mathbb{S}^{2n-1}$ be the vertices of a regular simplex in $\mathbb{R}^{2n}\times\{0\}\subseteq\mathbb{R}^{2n+1}$ inscribed in $\mathbb{S}^{2n-1}$. For $i=1,\dots,2n+1$, consider the Voronoi cells 
\[
V^{2n-1}_i:=\left\{
x\in\mathbb{S}^{2n-1};d_{\mathbb{S}^{2n}}(x,p_i)<d_{\mathbb{S}^{2n}}(x,p_j)\text{ for all }j\neq i
\right\},
\]
\[
V^{2n}_i:=\left\{
x\in H^{2n}_+;d_{\mathbb{S}^{2n}}(x,p_i)<d_{\mathbb{S}^{2n}}(x,p_j)\text{ for all }j\neq i
\right\}.
\]

\begin{figure}[ht]
    \centering
    \includegraphics[width=0.5\textwidth]{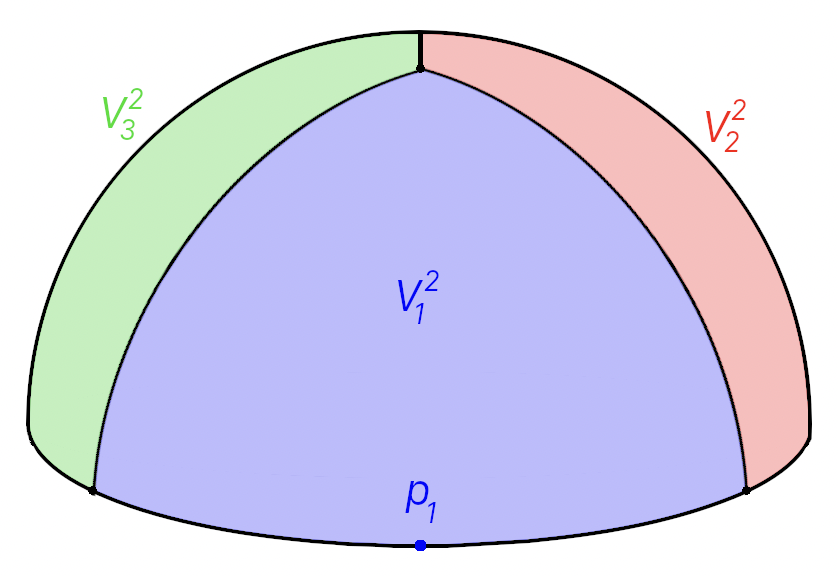}
    \hfill\includegraphics[width=0.4\textwidth]{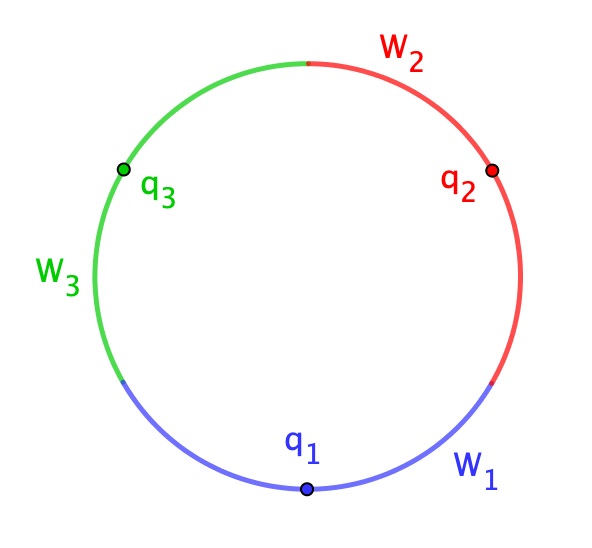}
    \caption{The Voronoi cells $V_i^{2n}$ and $W_i$ in the case $n=1$}
\end{figure}

Note that $V^{2n}_i$ is obtained by taking a cone (see \Cref{ConeOperation}) from $V^{2n-1}_i$ with respect to the point $(0,\dots,0,1)\in\mathbb{R}^{2n+1}$, so by Equation \ref{DiamCones} and \Cref{Viprops}\ref{Viprops545543} we have
\[
\text{diam}(V^{2n}_i)=\text{diam}(V^{2n-1}_i)=\arccos\left(\frac{-2n}{2n+2}\right)=\arccos\left(\frac{-n}{n+1}\right)
\]

Also, letting $q_1,\dots,q_{2n+1}$ be the vertices of a regular $(2n+1)$-gon inscribed in $\mathbb{S}^1$, we define the Voronoi cells
\[
W_i:=\left\{
y\in\mathbb{S}^1;d_{\mathbb{S}^{1}}(y,q_i)<d_{\mathbb{S}^{1}}(y,q_j)\text{ for all }j\neq i
\right\},
\]
which are intervals of length $\frac{2\pi}{2n+1}$.

We will need the fact that the diameter of $V^{2n}_i$ is  at most $\frac{2\pi n}{2n+1}$:
\begin{lemma}\label{Ineqarccosñsñslsad}
For all positive $x\in[1,\infty)$ we have $\arccos\left(\frac{-x}{x+1}\right)\leq\frac{2\pi x}{2x+1}$.
\end{lemma}
The following elegant proof of \Cref{Ineqarccosñsñslsad} is due to Pablo Vitoria.\footnote{This proposition plays a similar role to \cite[Lemma 4.2]{HJ}.}

\begin{proof}
Note that for any $x\in(1,\infty)$ we have
\[
\arccos\left(\frac{-x}{x+1}\right)\leq\frac{2\pi x}{2x+1}
\iff\frac{-x}{x+1}\geq\cos\left(\frac{2\pi x}{2x+1}\right)\iff\frac{x}{x+1}\leq\cos\left(\frac{\pi}{2x+1}\right)
\]
Changing variables to $y=\frac{\pi}{2x+1}$ (so that $\frac{x}{x+1}=\frac{2\pi}{\pi+y}-1$), it will be enough to check that $\cos(y)\geq\frac{2\pi}{\pi+y}-1$ for all $y\in\left[0,\frac{\pi}{3}\right]$. But this inequality follows from the facts that:
\begin{itemize}
    \item The functions $\cos(y)$ and $\frac{2\pi}{\pi+y}-1$ agree at $y=0,\frac{\pi}{3}$.
    \item In the interval $\left[0,\frac{\pi}{3}\right]$ the function $\cos(y)$ is concave while $\frac{2\pi}{\pi+y}-1$ is convex.\qedhere
\end{itemize}

\end{proof}

\begin{proof}[Proof of \Cref{S1toSeven}]
Consider the following relation $R_{2n}\subseteq H^{2n}_+\times\mathbb{S}^1$:
\begin{equation}
\label{CorrespSEven}
R_{2n}:=\left(\bigcup_{i=1}^{2n+1}V^{2n}_i\times\{q_i\}\right)
\bigcup
\left(\bigcup_{i=1}^{2n+1}\{p_i\}\times W_i\right).
\end{equation}
Note that the projection of $R_{2n}$ onto its first coordinate is dense in $H^{2n}_+$, because it is the union of the Voronoi cells $V_i^{2n}$ for $i=1,\dots,2n+1$. Similarly, the projection of $R_{2n}$ to its second coordinate is dense in $\mathbb{S}^1$. Thus, $R_{2n}\cup-R_{2n}$ is a metric correspondence between $\mathbb{S}^{2n}$ and $\mathbb{S}^1$. So thanks to the Helmet trick (\Cref{HelmetTrick}), we have
\[
d_{\text{GH}}(\mathbb{S}^{2n},\mathbb{S}^1)\leq\frac{1}{2}\text{dis}(R_{2n}\cup-R_{2n})=\frac{1}{2}\text{dis}(R_{2n}).
\]

So to prove \Cref{S1toSeven} it will be enough to prove that $\text{dis}(R_{2n})\leq\frac{2\pi n}{2n+1}$. That is, we need to prove that if $(x,y)$ and $(x',y')$ are in $R_{2n}$, then 
\begin{equation}\label{eq56445}
|d_{\mathbb{S}^{2n}}(x,x')-d_{\mathbb{S}^{1}}(y,y')|\leq\frac{2\pi n}{2n+1}.
\end{equation}
To prove \Cref{eq56445} we will divide the analysis into $6$ cases. 
\begin{itemize}
    \item $x,x'\in V^{2n}_i,y=y'=q_i$ for some $i\in\{1,\dots,2n+1\}$. In this case 
    \[|d_{\mathbb{S}^{2n}}(x,x')-d_{\mathbb{S}^{1}}(y,y')|=d_{\mathbb{S}^{2n}}(x,x')\leq\text{diam}(V^{2n}_i)=\arccos\left(\frac{-n}{n+1}\right)\leq\frac{2\pi n}{2n+1}.\]
    
    \item $x\in V^{2n}_i,x'\in V^{2n}_j$, $y=q_i,y'=q_j$ for some $i\neq j$. In this case we have $d_{\mathbb{S}^{1}}(y,y')\in\left[\frac{2\pi}{2n+1},\frac{2\pi n}{2n+1}\right]$, which together with $d_{\mathbb{S}^{2n}}(x,x')\in[0,\pi]$ implies \Cref{eq56445}.
    
    \item $x\in V^{2n}_i,x'=p_i,y=q_i,y'\in W_i$ for some $i$. Then $d_{\mathbb{S}^{1}}(y,y')\leq\frac{\pi}{2}$, and, as all points of $V_i^{2n-1}$ are at distance $\leq\frac{\pi}{2}$ of $p_i$, the same is true for all points of $V_i^{2n}$, so $d_{\mathbb{S}^{2n}}(x,x')=d_{\mathbb{S}^{2n}}(x,p_i)\leq\frac{\pi}{2}$, so $|d_{\mathbb{S}^{2n}}(x,x')-d_{\mathbb{S}^{1}}(y,y')|\leq\frac{\pi}{2}\leq\frac{2\pi n}{2n+1}$.
    
    \item $x\in V^{2n}_i,x'=p_j,y=q_i,y'\in W_j$ for some $i\neq j$. Then we have
    \[
    d_{\mathbb{S}^{1}}(y,y')=d_{\mathbb{S}^{1}}(q_i,y')>\frac{\pi}{2n+1}
    \]
    \[
    d_{\mathbb{S}^{2n}}(x,x')=d_{\mathbb{S}^{2n}}(x,p_j)>\frac{\zeta_{2n-1}}{2}=\frac{1}{2}\arccos\left(\frac{-1}{2n}\right)\text{ (see Proposition \ref{Viprops}\ref{Viprops986865}),}
    \]
    but for $n\geq1$ we have $\frac{1}{2}\arccos\left(\frac{-1}{2n}\right)\geq\frac{\pi}{2n+1}$, because we have equality for $n=1$ and for $n\geq2$ we have $\frac{1}{2}\arccos\left(\frac{-1}{2n}\right)\geq\frac{\pi}{4}\geq\frac{\pi}{2n+1}$. So both $d_{\mathbb{S}^{2n}}(x,x')$ and $d_{\mathbb{S}^{1}}(y,y')$ are in the interval $\left[\frac{\pi}{2n+1},\pi\right]$, thus $|d_{\mathbb{S}^{2n}}(x,x')-d_{\mathbb{S}^{1}}(y,y')|\leq\frac{2\pi n}{2n+1}$ in this case.

    \item $x=x'=p_i,y,y'\in W_i$ for some $i$. Then $d_{\mathbb{S}^{2n}}(x,x')=0$ and $d_{\mathbb{S}^{1}}(y,y')<\frac{2\pi}{2n+1}$ so we are done. 
    
    \item $x=p_i,x'=p_j$, $y\in W_i$, $y'\in W_j$ for some $i\neq j$. Then $d_{\mathbb{S}^{2n}}(x,x')=\zeta_{2n-1}\in\left[\frac{\pi}{2},\frac{2\pi}{3}\right]$, which together with $d_{\mathbb{S}^{2n}}(y,y')\in[0,\pi]$ implies \Cref{eq56445}.
    \qedhere
\end{itemize}

\end{proof}

\begin{remark}\label{RmkEvenMapHJ}
Our construction is similar to that of \cite[Sections 3,4]{HJ}. Both are extensions of the idea from \cite[Appendix D]{LMS} of considering finite subsets $\mathcal{P},\mathcal{Q}$ of $\mathbb{S}^{2n},\mathbb{S}^{1}$ of the same cardinality and using a bijection between $\mathcal{P}$ and $\mathcal{Q}$ to construct a correspondence between $\mathbb{S}^{2n}$ and $\mathbb{S}^1$. The main difference is that in our case we need the helmet trick (as the sets $V_i^{2n}$ are Voronoi cells inside $H_+^{2n}$, not $\mathbb{S}^{2n}$) and that we use the vertices of a $(2n+1)$-simplex as centers of the Voronoi cells in $\mathbb{S}^{2n}$, while in \cite{HJ} they use an orthonormal basis of vectors, together with their antipodals. 

However, in both our construction and the one in \cite{HJ}, the correspondence between $\mathbb{S}^1$ and $\mathbb{S}^{2n}$ has distortion at least $\frac{2\pi n}{2n+1}$ for the same reason: there are pairs of points in $\mathbb{S}^{2n}$ which are arbitrarily close to each other  which are mapped to points at distance $\frac{2\pi n}{2n+1}$ in $\mathbb{S}^1$. 
\end{remark}

\msection{Distance from $\mathbb{S}^1$ to odd dimensional spheres}\label{SecS1Sodd}
In this section, we present a new proof of \Cref{S1toSeven}.

\msubsection{An optimal correspondence between $\mathbb{S}^{2n+1}$ and $\mathbb{S}^1$}

Let us start with some notation which we will use in this section:
\[
\mathbb{S}^{2n+1}=\left\{x\in\mathbb{R}^{2n+2};\sum_{j=1}^{2n+2}x_j^2=1\right\},
\]
\[
\mathbb{S}^{2n}:=\left\{x\in\mathbb{S}^{2n+1};x_{2n+2}=0\right\},
\]
\[
\mathbb{S}^{2n-1}:=\left\{x\in\mathbb{S}^{2n+1};x_{2n+1}=x_{2n+2}=0\right\},
\]
\[
H^{2n+1}_{+}:=\{x\in\mathbb{S}^{2n+1};x_{2n+2}\geq0\},
\]
\[H^{2n}_+:=\{x\in\mathbb{S}^{2n};x_{2n+1}\geq0\},\]
\[
H^{2n+1}_{++}:=\{x\in\mathbb{S}^{2n+1};x_{2n+1}\geq0,x_{2n+2}\geq0\},
\]
\[
H^{2n+1}_{-+}:=\{x\in\mathbb{S}^{2n+1};x_{2n+1}\leq0,x_{2n+2}\geq0\}.
\]
These sets are illustrated in \Cref{S1SoddDomain}.

\begin{figure}[ht]
    \centering
    \includegraphics[width=0.6\textwidth]{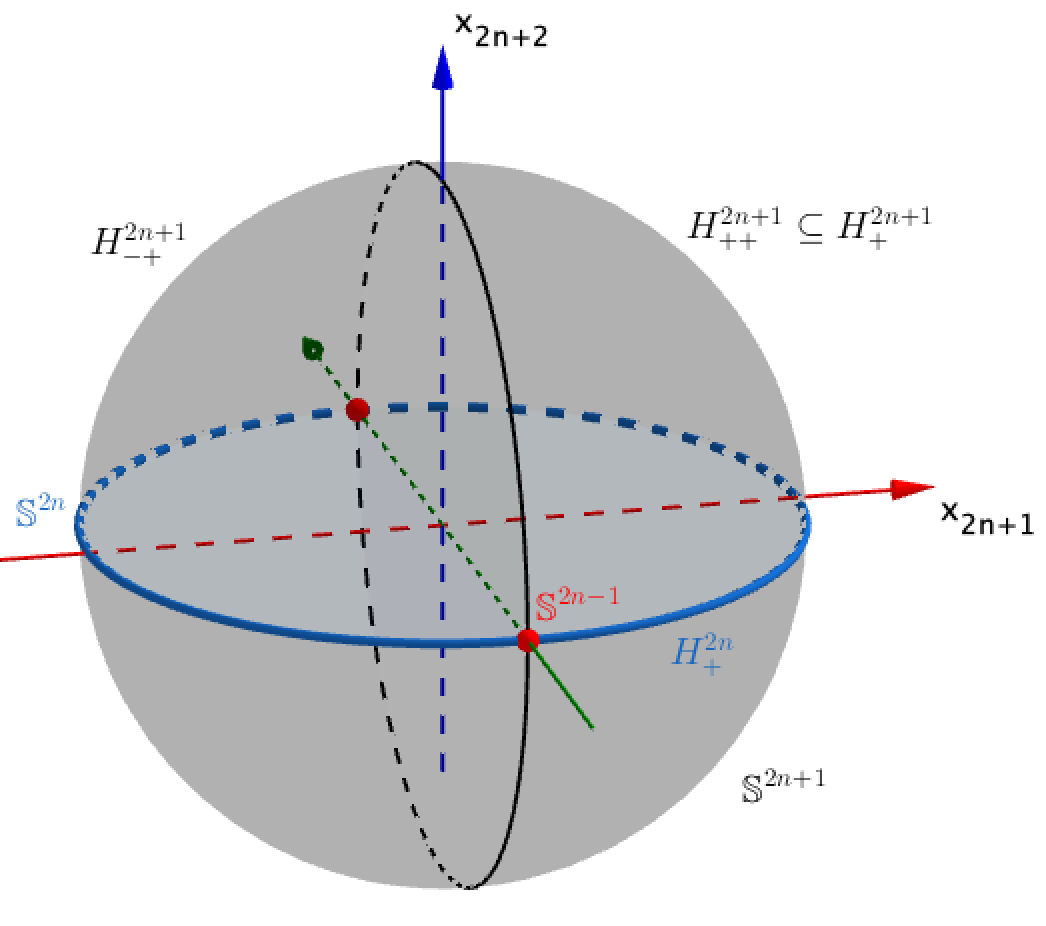}
    \caption{Some subsets of $\mathbb{S}^{2n+1}$ we will use in our construction.}
    \label{S1SoddDomain}
\end{figure}

Now, for a dense subset $D$ of $H_+^{2n+1}$, we will define a map $\Phi:D\to\mathbb{S}^1$ with distortion $\frac{2\pi n}{2n+1}$ and such that, if $G_\Phi$ is the graph of $\Phi$, then $G_\Phi\cup -G_\Phi\subseteq\mathbb{S}^{2n+1}\times\mathbb{S}^1$
is a metric correspondence between $\mathbb{S}^{2n+1}$ and $\mathbb{S}^1$. By the helmet trick (\Cref{HelmetTrick}) this will establish \Cref{S1toSodd}.

Before describing the map $\Phi$ in detail, we give a more informal description of it.

\begin{notation}\label{notationpalpha}
Let $N=(0,\dots,0,1)\in\mathbb{R}^{2n+2}$ be the north pole of $\mathbb{S}^{2n+1}$. We will denote points $x\in H_+^{2n+1}\setminus\{N\}$ as $(p,\alpha)\in\mathbb{S}^{2n}\times\left[0,\frac{\pi}{2}\right)$, where $p\in\mathbb{S}^{2n}\subseteq\mathbb{S}^{2n+1}$ is the point of $\mathbb{S}^{2n}$ closest to $x$ and $\alpha\in\left[0,\frac{\pi}{2}\right)$ is the geodesic distance from $p$ to $\mathbb{S}^{2n}$.
\end{notation}

Now, note that the metric correspondence $R_{2n}$ from \Cref{SectionS1Seven} was the union of the graphs of two odd maps, one map $f:\mathbb{S}^{2n}\to\{\pm q_0,\dots,\pm q_{2n}\}\subseteq\mathbb{S}^1$ and one map $g:\mathbb{S}^1\to\{\pm p_0,\dots,\pm p_{2n}\}\subseteq\mathbb{S}^{2n}$ (the domains of the maps are actually just dense subsets of $\mathbb{S}^{2n},\mathbb{S}^1$ respectively). Then, the restriction of our map $\Phi$ to $\mathbb{S}^{2n}$ will be just the map $f$: we let $\Phi(p,0)=f(p)$ for almost all $p\in\mathbb{S}^{2n}$.

In fact, if $p\in H^{2n}_+$ (so $p$ is in the northern hemisphere of $\mathbb{S}^{2n}$), we let $\Phi(p,\alpha)=f(p)$ for all $\alpha\in\left[0,\frac{\pi}{2}\right)$. That determines $\Phi(x)$ for almost all $x\in H^{2n+1}_{++}$. 
For points of $H^{2n+1}_{-+}$, that is, points of the form $(p,\alpha)$ where $p$ is in the southern hemisphere of $\mathbb{S}^{2n}$, we define $\Phi(p,\alpha)=f(p)\cdot e^{i\cdot\min\left(\alpha,\frac{\pi}{2n+1}\right)}\in\mathbb{S}^1$. 
So for fixed $p$, as $\alpha$ increases from $0$ to $\frac{\pi}{2n+1}$, $\Phi(p,\alpha)$ follows a unit speed geodesic in $\mathbb{S}^1$ from $f(p)$ to $f(p)\cdot e^{\frac{i \pi }{2n+1}}$, and then $\Phi(p,\alpha)$ is equal to $f(p)\cdot e^{\frac{i\pi }{2n+1}}$ for all $\alpha\in\left[\frac{\pi}{2n+1},\frac{\pi}{2}\right)$. 
We include a depiction of the map $\Phi$ in \Cref{S1vsS3Arrow} in the case when $2n+1=3$.

\begin{figure}[ht]
    \centering
    \includegraphics[width=0.6\textwidth]{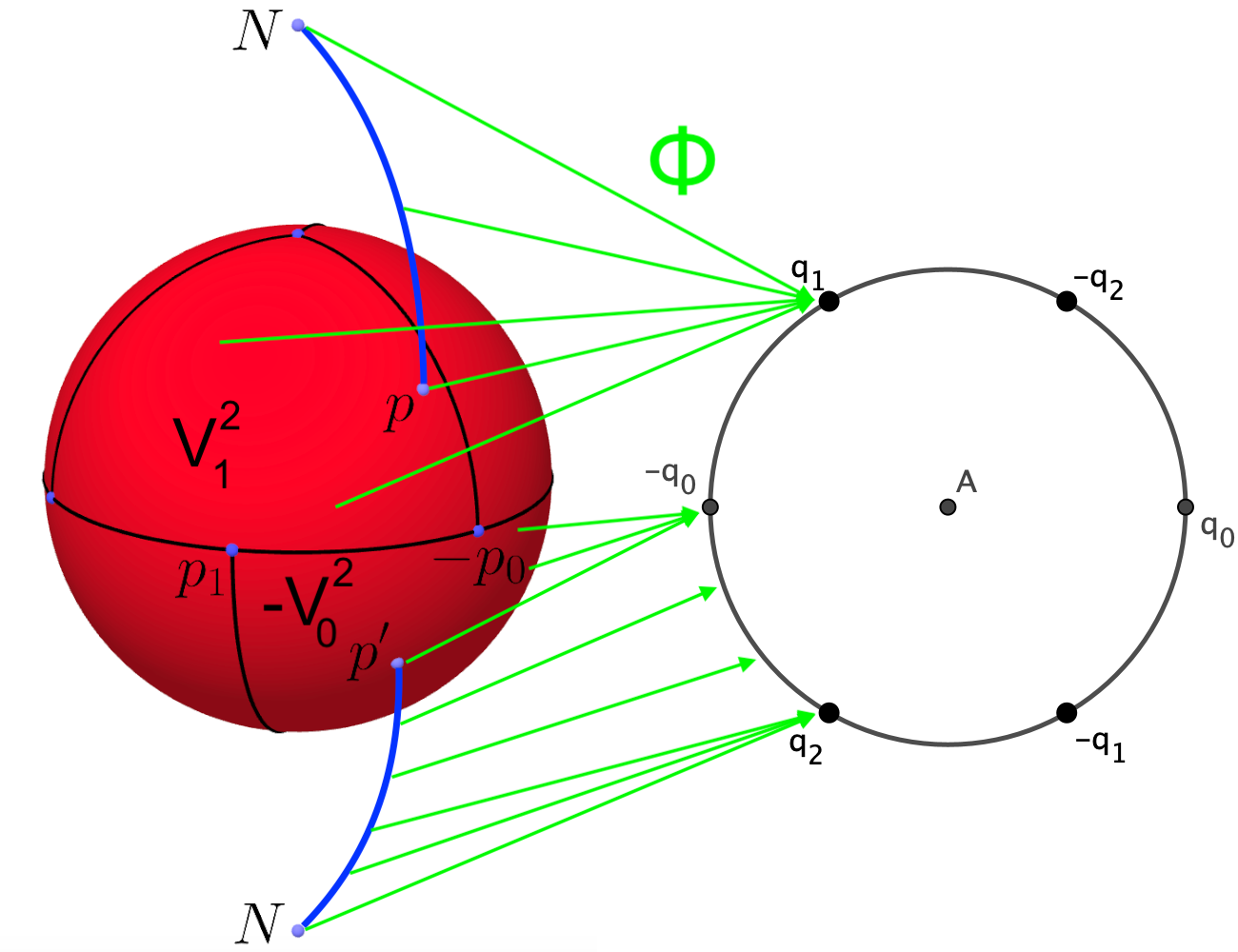}
    \caption{The map $\Phi:D\to\mathbb{S}^{1}$. This figure depicts the restriction of $\Phi$ to $\mathbb{S}^{2n}=\mathbb{S}^2$ and also its restriction to two geodesics (colored blue) between points of $\mathbb{S}^{2n}$ and $N$, the north pole of $\mathbb{S}^3$. We utilize two copies of $N$ for more clarity. Note that all  points in the geodesic segment $[p',N]$ with $\alpha>\frac{\pi}{2n+1}=\frac{\pi}{3}$ are mapped to $-q_0\cdot e^{\frac{i\pi}{2n+1}}=q_2$.}
    \label{S1vsS3Arrow}
\end{figure}

Now we will define the map $\Phi$ in detail. Let $p_0,\dots,p_{2n}$ be the vertices of a regular simplex in $\mathbb{R}^{2n}\times\{(0,0)\}$ inscribed in $\mathbb{S}^{2n-1}$, and for $j=0,\dots,2n$ we define the Voronoi cells
\[
V^{2n-1}_j:=\{x\in\mathbb{S}^{2n-1};d_{\mathbb{S}^{2n+1}}(x,p_j)<d_{\mathbb{S}^{2n+1}}(x,p_k)\text{ for all }k\neq j\},
\]
\[
V^{2n}_j:=\{x\in H^{2n}_+;d_{\mathbb{S}^{2n+1}}(x,p_j)<d_{\mathbb{S}^{2n+1}}(x,p_k)\text{ for all }k\neq j\},
\]
\[
V^{2n+1}_j=\{x\in H^{2n+1}_{++};d_{\mathbb{S}^{2n+1}}(x,p_j)<d_{\mathbb{S}^{2n+1}}(x,p_k)\text{ for all }k\neq j\}.
\]
Note that $V_j^{2n}=\left(C_{(0,\dots,1,0)}V_j^{2n-1}\right)\setminus\{(0,\dots,1,0)\}$, with $C$ being defined as in \Cref{ConeOperation}, and $V_j^{2n+1}=\left(C_{(0,\dots,0,1)}V_j^{2n}\right)\setminus\{(0,\dots,0,1)\}$. Thus Equation \ref{DiamCones} and \Cref{Viprops}\ref{Viprops545543} imply that, for all $j=1,\dots,2n+1$,
\begin{equation}\label{DiamViSodd}
\text{diam}(V_j^{2n+1})
=\text{diam}(V_j^{2n})
=\text{diam}(V_j^{2n-1})
=\arccos\left(\frac{-2n}{2n+2}\right).
\end{equation}

In this section, we identify $\mathbb{S}^1$ with $\{z\in\mathbb{C};|z|=1\}$, and for $k=0,\dots,2n$ we let $q_k=e^{\frac{2\pi ik}{2n+1}}$ denote the $(2n+1)$-th roots of unity.  Finally, we define 
\[
D=\left(\bigcup_{j=0}^{2n}V^{2n+1}_j\right)\cup\left(\bigcup_{j=0}^{2n}C_N(-V^{2n}_j)\setminus\{N\}\right).
\]
And consider the map
\[
\begin{array}{cccll}
\Phi:&D&\to&\mathbb{S}^1;&  \\
&(p,\alpha)&\mapsto&q_k&\text{ if }p\in V_k^{2n}.\\
&(p,\alpha)&\mapsto&-q_k\cdot e^{i\min\left(\alpha,\frac{\pi}{2n+1}\right)}&\text{ if }p\in-V_k^{2n}.\\
\end{array}
\]
For simplicity we will write $\Phi(p,\alpha)$ instead of $\Phi((p,\alpha))$ for the image of the point $(p,\alpha)$. Also note that the condition $p\in V_k^{2n}$ is equivalent to $(p,\alpha)\in V_k^{2n+1}$, and recall that $G_\Phi\subseteq\mathbb{S}^{2n+1}\times\mathbb{S}^1$ denotes the graph of $\Phi$.

\begin{prop}
The relation $G_\Phi\cup-G_\Phi\subseteq\mathbb{S}^{2n+1}\times\mathbb{S}^1$ is a metric correspondence (as defined in \Cref{trhgfdtrgf}) between $\mathbb{S}^{2n+1}$ and $\mathbb{S}^1$.
\end{prop}

\begin{proof}
Firstly, the domain $D$ of $\Phi$ is dense in $H^{2n+1}_+$. So $D\cup-D$, which is the projection of $G_\Phi\cup-G_\Phi$ to $\mathbb{S}^{2n+1}$, is dense in $\mathbb{S}^{2n+1}$. Secondly, the image of $\Phi$ is the following set $I$ (see \Cref{S1SoddIntervalsinS1} for the case $2n+1=5$)
\begin{equation}
I:=\bigcup_{k=0}^{2n}I_k\subseteq\mathbb{S}^1,
\end{equation}
where $I_k$ is the following interval of length $\frac{\pi}{2n+1}$ having $q_k$ as one endpoint:
\begin{equation}\label{DefIk}
I_k:=\left\{e^{ix};x\in\left[\frac{\pi(2k-1)}{2n+1},\frac{2\pi k}{2n+1}\right]\right\}\text{ for }k=0,1,\dots,2n.
\end{equation}

So $I\cup-I$, which is the projection of $G_\Phi\cup-G_\Phi$ to $\mathbb{S}^1$, is the entire $\mathbb{S}^1$, concluding the proof that $G_\Phi\cup-G_\Phi$ is a metric correspondence between $\mathbb{S}^{2n+1}$ and $\mathbb{S}^1$.
\end{proof}

\begin{figure}[ht]
    \centering
    \includegraphics[width=0.56\textwidth]{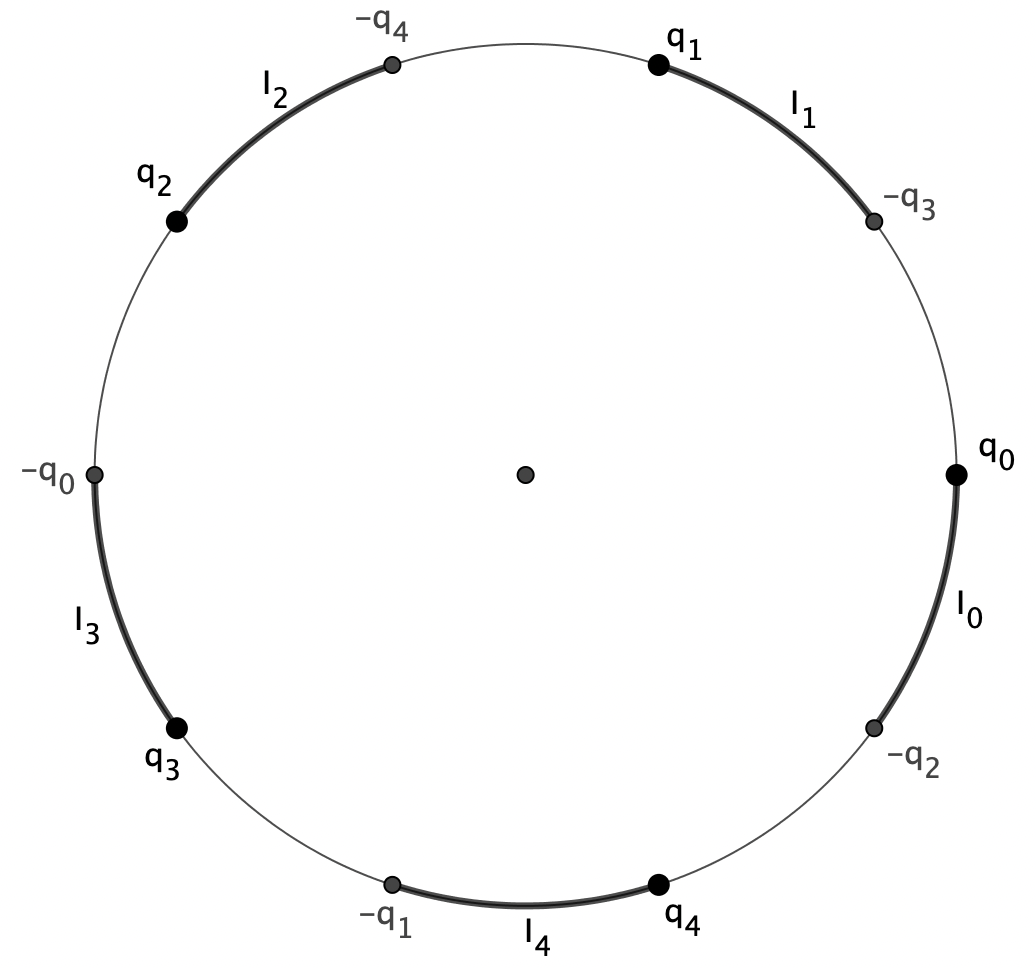}
    \caption{The image of $\Phi$ is half of $\mathbb{S}^1$ (case $n=2$)}
    \label{S1SoddIntervalsinS1}
\end{figure}

Note that the graph of the restriction of $\Phi$ to $\mathbb{S}^{2n}$ is
\begin{equation*}
\bigcup_{i=0}^{2n}\left(V_i^{2n}\times\{q_k\}\right)\cup\left(-V_i^{2n}\times\{-q_k\}\right),
\end{equation*}
which is isometric to a subset of the relation $R_{2n}\cup-R_{2n}$ (see \Cref{CorrespSEven}) that, as we proved in
\Cref{S1toSeven}, has distortion $\frac{2\pi n}{2n+1}$. That is, for any $(p,0),(p',0)\in D$ we have 
\begin{equation}\label{InductiveIneqlwklask}
|d_{\mathbb{S}^{1}}(\Phi(p,0),\Phi(p',0))-d_{\mathbb{S}^{2n+1}}((p,0),(p',0))|\leq\frac{2\pi n}{2n+1}.
\end{equation}

Also note that $\Phi$ maps most points of $D$ to $\{q_0,\dots,q_{2n}\}$; letting $A=\Phi^{-1}(\{q_0,\dots,q_{2n}\})$, the set $B:=D\setminus A$ of points of $D$ mapped by $\Phi$ outside of $\{q_0,\dots,q_{2n}\}$ is essentially the set of points of $H^{2n+1}_{-+}$ at distance $<\frac{\pi}{2n+1}$ of $\mathbb{S}^{2n}$: 
\[
B=D\setminus A
=\bigcup_{j=0}^{2n}\left\{(p,\alpha)\in-V^{2n+1}_j;\alpha\in\left[0,\frac{\pi}{2n+1}\right)\right\}.
\]

\begin{figure}[ht]
    \centering
    \includegraphics[width=0.6\textwidth]{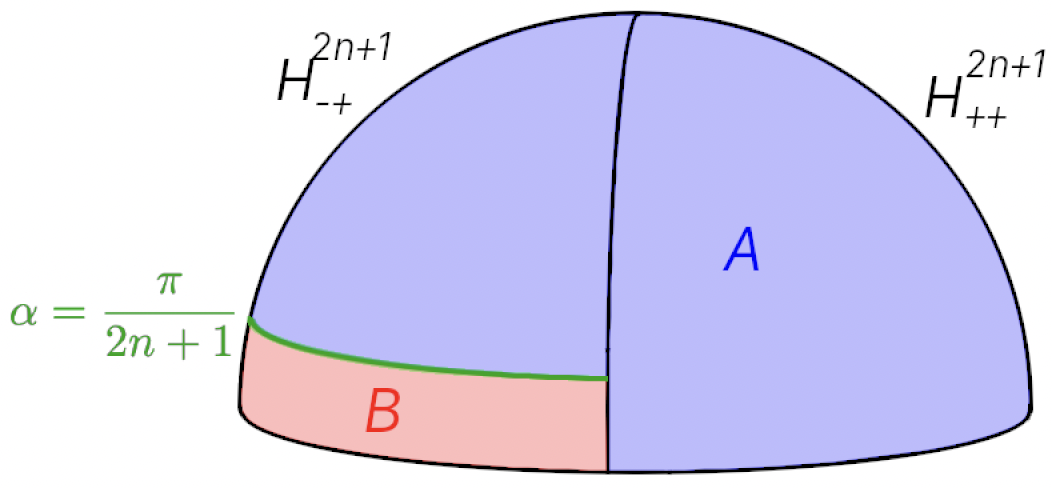}
    \caption{A depiction of the sets $A$ and $B:=D\setminus A$.}
    \label{FigureAandB}
\end{figure}

\msubsection{Proof that $\Phi$ has distortion $\frac{2\pi n}{2n+1}$}
We want to prove that, for any two points $(p,\alpha)$ and $(p',\alpha')$ in $D\subseteq H^{2n+1}_+$, we have
\[
|d_{\mathbb{S}^{1}}(\Phi(p,\alpha),\Phi(p',\alpha'))-d_{\mathbb{S}^{2n+1}}((p,\alpha),(p',\alpha'))|\leq\frac{2\pi n}{2n+1}.
\]
That means that neither of the following two inequalities can happen:
\begin{equation}\label{345ter}
d_{\mathbb{S}^{1}}(\Phi(p,\alpha),\Phi(p',\alpha'))-d_{\mathbb{S}^{2n+1}}((p,\alpha),(p',\alpha'))>\frac{2\pi n}{2n+1}.
\end{equation}

\begin{equation}\label{234rrrte}
d_{\mathbb{S}^{2n+1}}((p,\alpha),(p',\alpha'))-d_{\mathbb{S}^{1}}(\Phi(p,\alpha),\Phi(p',\alpha'))>\frac{2\pi n}{2n+1}.
\end{equation}

\begin{proof}
It follows from the fact that $t\mapsto d_{\mathbb{S}^k}(p,\gamma(t))$ is $1$-Lipschitz.
\end{proof}

\subsubsection{Why Inequality (\ref{345ter}) cannot happen}
We will assume that $(p,\alpha),(p',\alpha')\in D$ satisfy \Cref{345ter}
and obtain a contradiction.

Firstly, we may assume without loss of generality that $\Phi(p,\alpha)\in I_0$ ($I_0$ is defined in \Cref{DefIk}; see \Cref{S1SoddIntervalsinS1}). This, along with the fact that $d_{\mathbb{S}^{1}}(\Phi(p,\alpha),\Phi(p',\alpha'))>\frac{2n}{2n+1}\pi=\pi-\frac{\pi}{2n+1}$, implies that $\Phi(p',\alpha')$ is either in $I_n$ or in $I_{n+1}$. We can in fact assume $\Phi(p',\alpha')\in I_{n+1}$, swapping $(p,\alpha)$ and $(p',\alpha')$ if not. So we have
\[
\Phi(p,\alpha)\in I_0 \text{ (interval between $-q_n$ and $q_0$).}
\]
\[
\Phi(p',\alpha')\in I_{n+1}\text{ (interval between $-q_0$ and $q_{n+1}$).}
\]

Note that, as $d_{\mathbb{S}^{1}}(\Phi(p,\alpha),\Phi(p',\alpha'))>\frac{2\pi n}{2n+1}$, the point $\Phi(p',\alpha')$ cannot be exactly $q_{n+1}$, as the entire interval $I_0$ lies at distance $\leq\frac{2\pi n}{2n+1}$ from $q_{n+1}$. This implies that $(p',\alpha')\in B$, that is, $\alpha'<\frac{\pi}{2n+1}$ and $(p',\alpha')\in H^{2n+1}_{-+}$. Now consider the function 
\begin{equation*}
\begin{array}{cccl}
h_1:&\left[0,\frac{\pi}{2n+1}\right]&\to&\mathbb{R};\\
&t&\mapsto&d_{\mathbb{S}^{1}}(\Phi(p,\alpha),\Phi(p',t))-d_{\mathbb{S}^{2n+1}}((p,\alpha),(p',t)).
\end{array}
\end{equation*}
Then $h_1$ is decreasing: this is because we have $\Phi(p',t)=e^{i(\pi+t)}$, so $\frac{d}{dt}d_{\mathbb{S}^{1}}(\Phi(p,\alpha),\Phi(p',t))=-1$, while the function $t\mapsto d_{\mathbb{S}^{2n+1}}((p,\alpha),(p',t))$ is $1$-Lipschitz, owing to $t\mapsto(p',t)$ being a unit speed geodesic. So we conclude that 
\begin{equation}
\label{terfdst43e}
d_{\mathbb{S}^{1}}(\Phi(p,\alpha),\Phi(p',0))-d_{\mathbb{S}^{2n+1}}((p,\alpha),(p',0))\geq d_{\mathbb{S}^{1}}(\Phi(p,\alpha),\Phi(p',\alpha'))-d_{\mathbb{S}^{2n+1}}((p,\alpha),(p',\alpha'))>\frac{2\pi n}{2n+1}. 
\end{equation}
In other words, we can assume without loss of generality that $\alpha'=0$ in Inequality \ref{345ter}. Now, Inequality \ref{terfdst43e} implies that $d_{\mathbb{S}^{2n+1}}((p,\alpha),(p',0))<\frac{\pi}{2n+1}$, so we have 

\[
\alpha=d_{\mathbb{S}^{2n+1}}((p,\alpha),\mathbb{S}^{n})
\leq
d_{\mathbb{S}^{2n+1}}((p,\alpha),(p',0))
<\frac{\pi}{2n+1}.
\] 
We also have $d_{\mathbb{S}^{2n}}(p,p')\leq\frac{\pi}{2}$: if not, $d_{\mathbb{S}^{2n+1}}((p,t),(p',0))$ would be more than $\frac{\pi}{2}$ for all $t$, contradicting $d_{\mathbb{S}^{2n+1}}((p,\alpha),(p',0))<\frac{\pi}{2n+1}$. Thus,
\begin{equation}\label{65yrtgdftr}
d_{\mathbb{S}^{2n+1}}((p,0),(p',0))\leq d_{\mathbb{S}^{2n+1}}((p,\alpha),(p',0)).
\end{equation}

Now, we divide the analysis into $2$ cases according to the quadrant to which $(p,\alpha)$ belongs:
\begin{itemize}
    \item Suppose $(p,\alpha)\in H^{2n+1}_{++}$.
    Then $\Phi(p,\alpha)=\Phi(p,0)$, so by Inequality \ref{65yrtgdftr} we have
    \begin{multline*}
    d_{\mathbb{S}^{1}}(\Phi(p,0),\Phi(p',0))-d_{\mathbb{S}^{2n+1}}((p,0),(p',0))
    \geq d_{\mathbb{S}^{1}}(\Phi(p,\alpha),\Phi(p',0))-d_{\mathbb{S}^{2n+1}}((p,\alpha),(p',0))>
    \frac{2\pi n}{2n+1},
    \end{multline*}
    contradicting Inequality \ref{InductiveIneqlwklask}.

    \item Suppose $(p,\alpha)\in H^{2n+1}_{-+}$. We also know that $\alpha\in\left[0,\frac{\pi}{2n+1}\right]$. However, the function
    \[
\begin{array}{cccl}
h_2:&\left[0,\frac{\pi}{2n+1}\right]&\to&\mathbb{R};\\
&t&\mapsto&d_{\mathbb{S}^{1}}(\Phi(p,t),\Phi(p',0))-d_{\mathbb{S}^{2n+1}}((p,t),(p',0)),
\end{array}\]
    is increasing because $\frac{d}{dt}d_{\mathbb{S}^{1}}(\Phi(p,t),\Phi(p',0))=1$ and the function $d_{\mathbb{S}^{2n+1}}((p,t),(p',0))$ is $1$-Lipschitz, due to $t\mapsto(p,t)$ being unit speed geodesic. So we have
\[
h_2\left(\frac{\pi}{2n+1}\right)\geq h_2(\alpha)>\frac{2\pi n}{2n+1}.
\]
However, this implies that $d_{\mathbb{S}^{2n+1}}\left(\left(p,\frac{\pi}{2n+1}\right),(p',0)\right)<\frac{\pi}{2n+1}$, which is impossible because $\left(p,\frac{\pi}{2n+1}\right)$ is at distance $\frac{\pi}{2n+1}$ from $\mathbb{S}^{2n}$, and  $\mathbb{S}^{2n}$    contains the point $(p',0)$.
\end{itemize}

\subsubsection{Why Inequality (\ref{234rrrte}) cannot happen}
We will assume that $(p,\alpha),(p',\alpha')\in D$ satisfy Inequality \ref{234rrrte}
and obtain a contradiction.

Inequality \ref{234rrrte} implies that $d_{\mathbb{S}^{1}}(\Phi(p,\alpha),\Phi(p',\alpha'))<\frac{\pi}{2n+1}$, so $\Phi(p,\alpha)$ and $\Phi(p',\alpha')$ are in the same subinterval $I_k$ of $\mathbb{S}^1$ of length $\frac{\pi}{2n+1}$. We can assume that this interval is $I_0$, so it has  $q_0$ and $-q_n$ as endpoints. Also note that, if $N=(0,\dots,0,1)\in\mathbb{S}^{2n+1}$ is the north pole, then
\begin{align*}
\alpha+\alpha'&=(\pi/2-d_{\mathbb{S}^{2n+1}}(N,(p,\alpha)))+(\pi/2-d_{\mathbb{S}^{2n+1}}(N,(p',\alpha')))\\&
\leq\pi-d_{\mathbb{S}^{2n+1}}((p,\alpha),(p',\alpha'))
<\frac{\pi}{2n+1},  
\end{align*}
the last inequality being due to Inequality \ref{234rrrte}. So $\alpha,\alpha'<\frac{\pi}{2n+1}$. We consider three cases:
\begin{itemize}
    \item Both $(p,\alpha)$ and $(p',\alpha')$ are in $H^{2n+1}_{++}$. Then $(p,\alpha),(p',\alpha')$ are both in $V^{2n+1}_0$. So by \Cref{DiamViSodd} and \Cref{Ineqarccosñsñslsad} we have
    \[d_{\mathbb{S}^{2n+1}}((p,\alpha),(p',\alpha'))<\text{diam}(V^{2n+1}_0)=\arccos\left(\frac{-2n}{2n+2}\right)
    =\arccos\left(\frac{-n}{n+1}\right)\leq\frac{2n}{2n+1}\pi,\]
    contradicting Inequality \ref{234rrrte}.

    \item Both $(p,\alpha)$ and $(p',\alpha')$ are in $H^{2n+1}_{-+}$. Then $(p,\alpha),(p',\alpha')$ are both in $-V_n^{2n+1}$, so as in the previous case we have $d_{\mathbb{S}^{2n+1}}((p,\alpha),(p',\alpha'))<\frac{2n}{2n+1}\pi$,
    contradicting Inequality \ref{234rrrte}.

    \item We have $(p,\alpha)\in H^{2n+1}_{++}$ and $(p',\alpha')\in H^{2n+1}_{-+}$. So $\Phi(p,\alpha)=q_0$ and, as $\alpha'<\frac{\pi}{2n+1}$, we have $\Phi(p',\alpha')=-q_ne^{i\alpha'}$. Now, the function
\[
\begin{array}{cccl}
h_3:&\left[0,\frac{\pi}{2n+1}\right]&\to&\mathbb{R};\\
&t&\mapsto&d_{\mathbb{S}^{2n+1}}((p,\alpha),(p',t))-d_{\mathbb{S}^{1}}(\Phi(p,\alpha),\Phi(p',t)).
\end{array}\]
    is increasing, because $t\mapsto d_{\mathbb{S}^{2n+1}}((p,\alpha),(p',t))$ is $1$-Lipschitz (due to the fact that $t\mapsto(p',t)$ is a unit speed geodesic) and $\frac{d}{dt}d_{\mathbb{S}^{1}}(\Phi(p,\alpha),\Phi(p',t))=-1$.

    So we have that $h_3\left(\frac{\pi}{2n+1}\right)\geq h_3(\alpha)>\frac{2n}{2n+1}\pi$. Which implies that $d_{\mathbb{S}^{2n+1}}\left((p,\alpha),\left(p',\frac{\pi}{2n+1}\right)\right)>\frac{2n}{2n+1}\pi$. This, however, contradicts the fact that
    \begin{align*}
    d_{\mathbb{S}^{2n+1}}\left((p,\alpha),\left(p',\frac{\pi}{2n+1}\right)\right)&\leq d_{\mathbb{S}^{2n+1}}\left((p,\alpha),N\right)+d_{\mathbb{S}^{2n+1}}\left(N,\left(p',\frac{\pi}{2n+1}\right)\right)\\
    &\leq\frac{\pi}{2}+\left(\frac{\pi}{2}-\frac{\pi}{2n+1}\right)=\frac{2n}{2n+1}\pi.
    \end{align*}
\end{itemize}

\msection{Distance from $\mathbb{S}^3$ to $\mathbb{S}^4$}\label{SecDS3S4}
In this section we prove that $d_{\text{GH}}(\mathbb{S}^3,\mathbb{S}^4)
=\frac{1}{2}\zeta_3$ by constructing a surjective function $F_n:\mathbb{S}^{n+1}
\to\mathbb{S}^n$ and proving that for $n=3$ it has distortion $\zeta_3$. The construction we utilize is inspired in the proof of \cite[Proposition 1.16]{LMS}.
We suspect that $F_n$ also has distortion $\zeta_n$ for more values of $n$ (e.g. $n=4,5$, see \Cref{CheckingS3S4}), but we have verified that $F_n$ has distortion $>\zeta_n$ for $n\geq7$ (see \Cref{SecDifIneq}). Let us first introduce some notation. 

\begin{itemize}
    \item We identify $\mathbb{S}^n$ with $\mathbb{S}^n\times\{0\}\subseteq\mathbb{S}^{n+1}\subseteq\mathbb{R}^{n+2}$, and we let $e_{n+2}$ be the north pole $(0,\dots,0,1)\in\mathbb{R}^{n+2}$. 
    \item For any $x,x'\in\mathbb{S}^{n+1}$ such that $x'\neq-x$ and $\lambda\in[0,1]$, we let $\lambda x\oplus(1-\lambda)x'$ denote the point $z$ in the geodesic segment $[x,x']$ such that $d_{\mathbb{S}^{n+1}}(z,x')=\lambda d_{\mathbb{S}^{n+1}}(x,x')$. 
    \item Let $\sigma:\mathbb{S}^{n+1}\setminus\{e_{n+2},-e_{n+2}\}\to\mathbb{S}^n$ be the projection to $\mathbb{S}^n$; that is, for each $x$ in the domain, $\sigma(x)$ will be the point of $\mathbb{S}^{n}$ closest to $x$.
    \item $p_1,\dots,p_{n+2}$ will be the vertices of a fixed regular simplex inscribed in $\mathbb{S}^n$.
    \item For each $x\in\mathbb{S}^{n+1}$ let $\alpha(x):=d_{\mathbb{S}^{n+1}}(x,e_{n+2})$.
    \item For each $i=1,\dots,n+2$, let 
    \[
    V_i:=
    \{x\in\mathbb{S}^{n};d_{\mathbb{S}^{n}}(p_i,x)<
    d_{\mathbb{S}^{n}}(p_j,x)\text{ for all }j\neq i\}.
    \]
    \item For each $i=1,\dots,n+2$ let $N_i=(C_{e_{n+2}}V_i)\setminus\{e_{n+2}\}$.  The sets $N_i$ are convex.
\end{itemize}

Note that $\bigcup_{i=1}^{n+2}N_i$ is dense in $\{x\in\mathbb{S}^{n+1};x_{n+2}\geq0\}$. We define the function $$F_n:\bigcup_{i=1}^{n+2}N_i\to\mathbb{S}^n$$ by 
\begin{equation}\label{DefFn}
F_n(x)=(1-f(\alpha(x)))p_i\oplus f(\alpha(x))\sigma(x), \text{ if $x$ in $N_i$,}
\end{equation}

where $f(t):=\max(0,t+1-\frac{\pi}{2})$, see \Cref{Figuref}.

\begin{figure}[ht]
    \centering
    \includegraphics[width=0.35\linewidth]{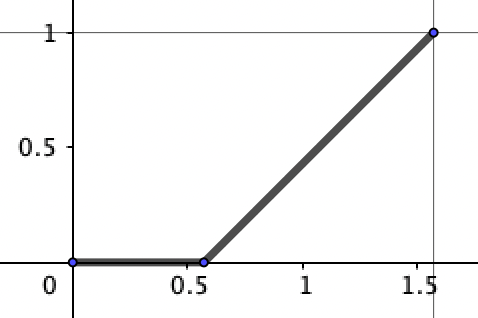}
    \caption{The function $f:\left[0,\pi/2\right]\to[0,1]$.}
    \label{Figuref}
\end{figure}

\begin{figure}[ht]
    \centering
    \includegraphics[width=0.5\textwidth]{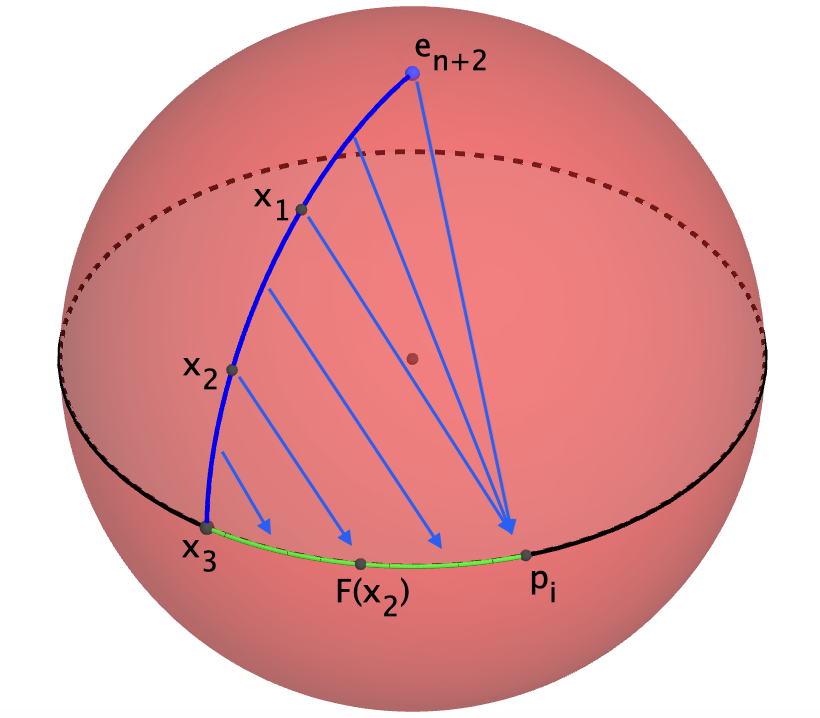}
    \caption{This figure represents the restriction of $F_n$ to a segment $[e_{n+2},x_3]$, for some $x_3\in V_i$. All points $x$ in the segment  $[e_{n+2},x_1]$ (where $\alpha(x_1)=\frac{\pi}{2}-1$) satisfy $F(x)=p_i$, while $F$ maps the segment $[x_1,x_3]$ `linearly' to  $[p_i,x_3]$.}
    \label{ExampleF}
\end{figure}

Now, it would be enough to prove that $\text{dis}(F_n)=\zeta_n$ in order to prove that $d_{\text{GH}}(\mathbb{S}^n,\mathbb{S}^{n+1})\leq\frac{\zeta_n}{2}$. Indeed, if $\text{dis}(F_n)=\zeta_n$, then the function
\[
F_n':\left(\bigcup_{i=1}^{n+2}N_i\right)\bigcup\left(-\bigcup_{i=1}^{n+2}N_i\right)
\to\mathbb{S}^n
\] defined by $F_n'(x)=-F_n'(-x)=F_n(x)$ for all $x\in\bigcup_{i=1}^{n+2}N_i$ also has distortion $\zeta_n$ by \Cref{HelmetTrick}, and its graph is a metric correspondence (as defined in \Cref{trhgfdtrgf}) between $\mathbb{S}^n$ and $\mathbb{S}^{n+1}$. Thus, $d_{\text{GH}}(\mathbb{S}^n,\mathbb{S}^{n+1})\leq\frac{\zeta_n}{2}$, as we wanted.

For simplicity we will write $d$ instead of $d_{\mathbb{S}^n}$ or $d_{\mathbb{S}^{n+1}}$ during the remainder of \Cref{SecDS3S4}, and we write $F:{H}^4_+\to\mathbb{S}^3$ instead of $F_3$.

The distortion of $F$ being at most $\zeta_3$ means that for all $x,x'$ in $\bigcup_{i=1}^{5}N_i$ we have 
\begin{equation}\label{IneqGen}
|d(x,x')-d(F(x),F(x'))|\leq\zeta_3.
\end{equation}

Depending on whether $x,x'$ are in one cone $N\in\{N_1,\dots,N_{5}\}$ or in different ones $N,N'$ and depending on whether $d(x,x')-d(F(x),F(x'))$ is positive or negative, \Cref{IneqGen} turns into four different inequalities, which give names to the following subsections.

\begin{remark}\label{CheckingS3S4}
To experimentally check whether the distortion of $F$ was $\zeta_3$, we used a python program (see \texttt{Random\_S3vsS4\_Check.py} in \cite{Git}) which chooses a finite set $S\subseteq\mathbb{S}^{4}$ of random points in the sphere, choosing the coordinates of the points according to a Gaussian distribution, and computes the distortion $|d(x,x')-d(F(x),F(x'))|$ for all $x,x'\in S$. 
Thanks to Daniel Hurtado for optimizing the program so that it could handle sets $S$ with hundreds of thousands of points.

The program was also useful for determining a function $f:\left[0,\frac{\pi}{2}\right]\to[0,1]$ which, when substituted in \Cref{DefFn}, would lead to $F$ having distortion $\zeta_3$. For example, if instead of the function $f$ from \Cref{Figuref} we chose $f(x)=\frac{2x}{\pi}$ (this was actually what we first tried), then we would have $\textup{dis}(F_3)>\zeta_3$.
\end{remark}

From now we set $n=3$ (we still sometimes write $n$ as some arguments may be generalized for all $n$). Before starting the division in cases, we prove a lemma which will be used in two of the cases.

For brevity, we let $p,p'\in\{p_1,\dots,p_{n+2}\}$ be the centers of the cones $N,N'$ respectively (so that $p=p_i,N=N_i,p'=p_j$ and $N'=N_j$ for some $i,j$), and we also denote
\begin{align*}
\alpha&:=\alpha(x)=d(x,e_{n+2})\\
\alpha'&:=\alpha(x')=d(x',e_{n+2})
\end{align*}

\begin{lemma}[Upper bound for $\alpha+\alpha'$]\label{UBx} We have $\alpha+\alpha'\leq2$.
\end{lemma}

\begin{proof}
We have $\alpha+\alpha'\geq\pi-2$ simply because in this part of the proof we assume $\alpha,\alpha'>\frac{\pi}{2}-1$.

Note that $d(x,F(x))+d(x',F(x'))$ has to be $>\zeta_3$ due to the triangle inequality and the fact that $|d(F(x),F(x'))-d(x,x')|>\zeta_3$. This cannot happen if $\alpha,\alpha'$ are very close to $\frac{\pi}{2}$; let us quantify this observation.

The triangle with vertices $x,\sigma(x),F(x)$ has a right angle at $\sigma(x)$ and sides $d(x,\sigma(x))=\frac{\pi}{2}-\alpha$ and $d(\sigma(x),F(x))\leq(\pi-\zeta_3)\left(\frac{\pi}{2}-\alpha\right)$, so we see by the cosine rule that
\[
d(x,F(x))
\leq
A(\alpha):=
\arccos\left(\cos\left(\frac{\pi}{2}-\alpha\right)\cos\left((\pi-\zeta_3)\left(\frac{\pi}{2}-\alpha\right)\right)\right)
\]
But by \Cref{Axconvex} below applied to $A\left(\frac{\pi}{2}-\alpha\right)$, the function $A(\alpha)$ is concave in the interval $\left(\frac{\pi}{2}-1,\frac{\pi}{2}\right)$. So for any fixed value of $\alpha+\alpha'$ we will have
\begin{equation*}
\zeta_3\leq
d(x,F(x))+d(x',F(x'))
\leq
A(\alpha)+A(\alpha')\leq2A\left(\frac{\alpha+\alpha'}{2}\right)
\end{equation*}
So $A(\frac{\alpha+\alpha'}{2})\geq\frac{\zeta_3}{2}$. So $\cos(A(\frac{\alpha+\alpha'}{2}))\leq\cos(\zeta_3/2)=\frac{\sqrt{3}}{2\sqrt{2}}$.

So to check $\alpha+\alpha'<2$, it is enough to check that, for all $x\in\left[0,\frac{\pi}{2}\right]$, 
$\sin(x)\cos\left((\pi-\zeta_3)\left(\frac{\pi}{2}-x\right)\right)\leq\frac{\sqrt{3}}{2\sqrt{2}}$ implies $x<1$. Equivalently, for all $y=\frac{\pi}{2}-x$ in $\left[0,\frac{\pi}{2}\right]$, $G(y):=\cos(y)\cos(y(\pi-\zeta_3))\leq\frac{\sqrt{3}}{2\sqrt{2}}$ implies $y>\frac{\pi}{2}-1$. But the function $G$ is decreasing in $\left[0,\frac{\pi}{2}-1\right]$, as $\cos(x)$ is positive and decreasing in $\left[0,\frac{\pi}{2}\right]$, and $G\left(\frac{\pi}{2}-1\right)\sim0.614>0.612\sim\frac{\sqrt{3}}{2\sqrt{2}}$, so we are done.
\end{proof}

\msubsection{$x,x'$ in the same cone $N$; $d(F(x),F(x'))\leq d(x,x')+\zeta_3$.}
We prove the stronger inequality $d(F(x),F(x'))\leq d(x,x')+\frac{\pi}{2}$.\footnote{This inequality holds for all $n$, not only $n=3$. An alternative way to prove $d(F(x),F(x'))\leq d(x,x')+\frac{\pi}{2}$ is checking that the function $F$ is $2$-Lipschitz when restricted to the Voronoi cell $N$; we then immediately have $d(F(x),F(x'))- d(x,x')\leq\frac{1}{2}d(F(x),F(x'))\leq\frac{\pi}{2}$.}

Suppose that $d(F(x),F(x'))>d(x,x')+\frac{\pi}{2}$ for the sake of contradiction. 

Consider \Cref{FigSmol435} below, in which the segments $[p,\sigma(x)]$ and $[p,\sigma(x')]$ have length $<\frac{\pi}{2}$. 

\begin{figure}[ht]
    \centering
    \includegraphics[width=0.6\linewidth]{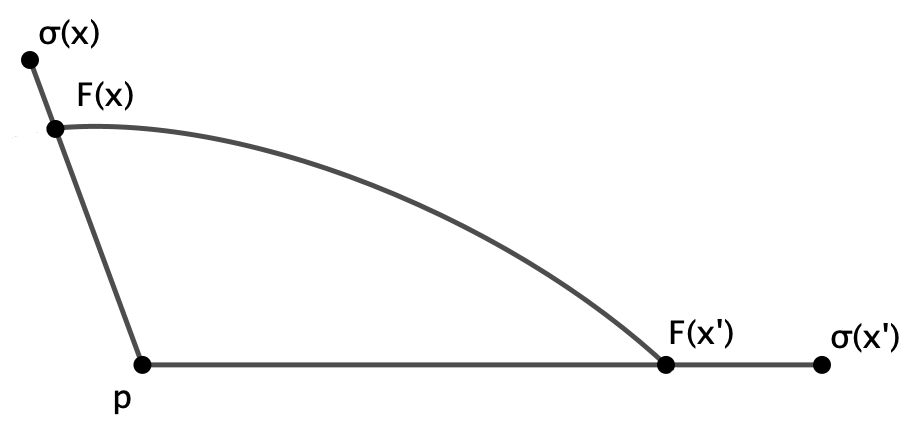}
    \caption{}
    \label{FigSmol435}
\end{figure}

As $d(F(x),F(x'))>\frac{\pi}{2}$, the angle $\angle F(x)pF(x')$ is $>\frac{\pi}{2}$ (see \Cref{54tret5regdf}). Thus,
\begin{equation*}
d(\sigma(x),\sigma(x'))\geq d(F(x),F(x'))\geq d(x,x')+\frac{\pi}{2}.
\end{equation*}
where the first inequality follows from \Cref{ytrgftgfcb}. So by the triangle inequality, $d(x,\sigma(x))+d(x',\sigma(x'))\geq d(\sigma(x),\sigma(x'))-d(x,x')\geq\frac{\pi}{2}$. That is, $\alpha+\alpha'=\pi-d(x,\sigma(x))+d(x',\sigma(x'))\leq\frac{\pi}{2}$. Which leads to contradiction:
\begin{multline*}
d(F(x),F(x'))\leq d(F(x),p)+d(F(x'),p)= f(\alpha)d(\sigma(x),p)+f(\alpha')d(\sigma(x'),p)\\\leq\frac{\pi}{2}(f(\alpha)+f(\alpha'))\leq\frac{\pi}{2},
\end{multline*}
where the last inequality uses that the condition $\alpha+\alpha'\leq\frac{\pi}{2}$ implies $f(\alpha)+f(\alpha')\leq1$, as $f(x)\leq\frac{2x}{\pi}$ (see \Cref{Figuref}).

\msubsection{$x,x'$ in the same cone $N$; $d(x,x')\leq d(F(x),F(x'))+\zeta_3$.}
Suppose that $d(x,x')> d(F(x),F(x'))+\zeta_3$ for the sake of contradiction. Note that by \Cref{ytrgftgfcb} applied to the triangle with vertices $A=e_{n+2},B=x$ and $C=x'$, the angle $\angle xe_{n+2}x'$ is $>\frac{\pi}{2}$, so applying again \Cref{ytrgftgfcb} to the triangle with vertices $A=e_{n+2},B=\sigma(x),C=\sigma(x')$, with the points $B',C'$ from \Cref{ytrgftgfcb} being $x,x'$ respectively, we have $d(\sigma(x),\sigma(x'))\geq d(x,x')>d(F(x),F(x'))+\zeta_3$. Thus,
\begin{multline*}
(\pi-\zeta_3)\left(\frac{\pi}{2}-\alpha\right)+(\pi-\zeta_3)\left(\frac{\pi}{2}-\alpha'\right)
\geq d(p,\sigma(x))(1-f(\alpha))+
d(p,\sigma(x'))(1-f(\alpha'))
\\= d(F(x),\sigma(x))+d(F(x'),\sigma(x'))\geq d(\sigma(x),\sigma(x'))-d(F(x),F(x'))>\zeta_3,
\end{multline*}
Here, the first inequality uses that $1-f(t)\leq\frac{\pi}{2}-t$ for all $t\in\left[0,\frac{\pi}{2}\right]$ and that $d(p,\sigma(x))\leq\pi-\zeta_3$ (by the definition of $F$ and \Cref{Viprops} \Cref{Viprops986865}). The equality is due to the definition of $F$ and after that we used the triangle inequality. We conclude from the inequality above that $\alpha(x)+\alpha(x')\leq\pi-\frac{\zeta_3}{\pi-\zeta_3}<1.76$. That cannot happen because 
\begin{equation*}
\alpha+\alpha'=d(e_{n+2},x)+d(e_{n+2},x')>d(x,x')>\zeta_3>1.82.
\end{equation*}

\msubsection{$x,x'$ in different cones $N\neq N'$; $d(x,x')\leq d(F(x),F(x'))+\zeta_3$.}
We prove $d(x,x')\leq d(F(x),F(x'))+\frac{\pi}{2}$ (which holds for all $F_n$, not only $F=F_3$, by changing $\zeta_3$ to $\zeta_n$ in the proof below). Suppose for the sake of contradiction that $d(x,x')> d(F(x),F(x'))+\frac{\pi}{2}$, and let $$H':=\{x\in\mathbb{S}^n;d(p',x)<d(p,x)\}\,\,\text{and}\,\,
H:=\{x\in\mathbb{S}^n;d(p,x)<d(p',x)\}.$$ We divide the analysis into $2$ cases.

\begin{enumerate}[label=\textbf{Case \arabic*.}, wide, labelwidth=0pt, labelindent=0pt]
    \item $\alpha,\alpha'>\frac{\pi}{2}-1$. As $\sigma(x)\in H$ and $\sigma(x')\in H'$, by \Cref{jensen>pi/2} and using that $d(p,p')=\zeta_3>\frac{\pi}{2}$ we get 
        \[
        d(F(x),H')=d\left(\left(\frac{\pi}{2}-\alpha\right)p\oplus \left(1+\alpha-\frac{\pi}{2}\right)\sigma(x),H'\right)\geq\frac{\pi}{4}\left(\frac{\pi}{2}-\alpha\right)
        \]
        \[
        d(F(x'),H)
        =d\left(\left(\frac{\pi}{2}-\alpha'\right)p'\oplus \left(1+\alpha'-\frac{\pi}{2}\right)\sigma(x'),H\right)
    \geq\frac{\pi}{4}\left(\frac{\pi}{2}-\alpha'\right).
        \]
        Note that the geodesic from $F(x)$ to $F(x')$ passes through the boundary $\partial H=\partial H'$ at some point $q_0$, thus
        \[
        d(F(x),F(x'))=d(F(x),q_0)+d(F(x'),q_0)\geq d(F(x),H')+d(F(x'),H)\geq
        \frac{\pi}{4}(\pi-\alpha-\alpha').
        \]
And now, using the triangle inequality, 
\[
\alpha+\alpha'\geq d(x,x')
\geq
\frac{\pi}{2}+d(F(x),F(x'))
\geq
\frac{\pi}{2}+\frac{\pi}{4}(\pi-\alpha-\alpha'),
\]
So $\alpha+\alpha'\geq\frac{\pi}{2}\frac{1+\frac{\pi}{2}}{1+\frac{\pi}{4}}>2.26$. This contradicts the proof of \Cref{UBx} (if in \Cref{UBx} we use the weaker hypotheses $|d(x, x')-d(F (x), F (x'))| > \frac{\pi}{2}$ instead of $|d(x, x')-d(F (x), F (x'))| >\zeta_3$, we still obtain via the same reasoning the weaker conclusion $\alpha+\alpha'<2.2$).

\item $\alpha\leq\frac{\pi}{2}-1$. Then $F(x)=p$, so $d(F(x),F(x'))\geq d(p,H')=\zeta_3/2\geq\frac{\pi}{4}$, so $d(x,x')\geq d(F(x),F(x'))+\frac{\pi}{2}\geq\frac{3\pi}{4}$, contradicting $d(x,x')\leq\alpha+\alpha'\leq2$ (again by \Cref{UBx}).
\end{enumerate}

\msubsection{$x,x'$ in different cones $N\neq N'$; $d(F(x),F(x'))\leq d(x,x')+\zeta_3$.}\label{SecDifIneq}
This is the inequality which is most difficult to prove, and it fails when instead of $n=3$ we have $n\geq7$; such a failure of the inequality can be reached when $\sigma(x)=-p'$,\footnote{$\sigma(x)$ cannot be exactly $-p'$, but it can be arbitrarily close to it.} $\sigma(x')=-p$ and $\alpha=\alpha'=\frac{\pi-1}{2}$, as in \Cref{FailFig}. Then we have $d(F(x),F(x'))=\pi$ and by the cosine rule applied to the triangle with vertices $e_{n+2},x,x'$, we have
\begin{equation*}
d(x,x')=\arccos\left(\cos\left(\frac{\pi-1}{2}\right)^2-\frac{\sin\left(\frac{\pi-1}{2}\right)^2}{n+1}\right)
\stackrel{n\to\infty}{\longrightarrow}
\arccos\left(\cos\left(\frac{\pi-1}{2}\right)^2\right)\sim1.339.
\end{equation*}
One can then check that for $n\geq7$, $d(x,x')+\zeta_n<\pi=d(F(x),F(x'))$.

\begin{figure}[ht]
    \centering
    \includegraphics[width=0.3\linewidth]{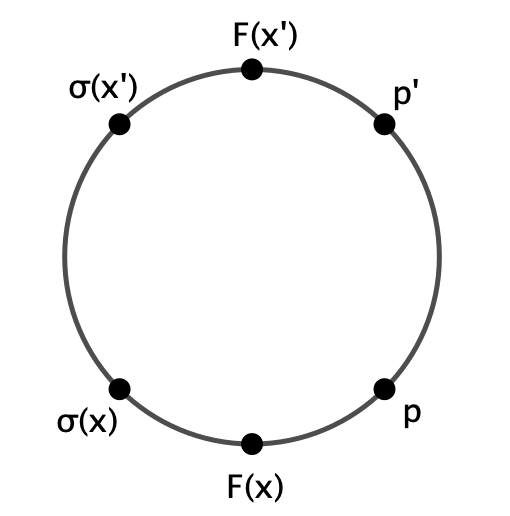}
    \caption{A case where the inequality fails for $n\geq7$.}
    \label{FailFig}
\end{figure}

In order to proceed, suppose for the sake of reaching a contradiction that $d(F(x),F(x'))>d(x,x')+\zeta_3$. We define the following constant, which we will be extremely important during the rest of the section:
\begin{equation*}
\kappa:=d(\sigma(x),\sigma(x'))
\end{equation*}

We first assume that $\alpha,\alpha'>\frac{\pi}{2}-1$. This case will comprise most of the proof; at the end we will study the simpler case where $\alpha$ or $\alpha'$ are at most $\frac{\pi}{2}-1$.

The proof will need computer assistance. We first prove that $\kappa\in[0.7,1.4]$ and $\alpha+\alpha'\in\left[\pi-2,2\right]$ in \Cref{Boundsxy}. Then we give some upper bounds for $d(F(x),F(x'))$ 
(the functions $\mathcal{U}_1,\mathcal{U}_2,\mathcal{U}_6$ introduced below) and a lower bound for $d(x,x')$ (the function $\mathcal{L}$ defined below) which only depend on $\alpha+\alpha'$ and $\kappa$. Then we check that $(\min(\mathcal{U}_1,\mathcal{U}_2,\mathcal{U}_6))(\alpha+\alpha',\kappa)$ cannot be bigger than $\mathcal{L}(\alpha+\alpha',\kappa)+\zeta_3$ for all $(\alpha+\alpha',\kappa)$ in the rectangle $[0.7,1.4]\times[\pi-2,2]$, concluding the proof.

\subsubsection{Lower bound for $d(x,x')$}

\begin{lemma}[Lower bound $\mathcal{L}$ for $d(x,x')$]\label{PropLBdx}
\begin{align*}
d(x,x')\geq \mathcal{L}(\alpha+\alpha',\kappa)&:=
\arccos\left(\cos(\alpha+\alpha')+\sin\left(\frac{\alpha+\alpha'}{2}\right)^2(1+\cos(\kappa))\right)\\
&=
\arccos\left(\cos(\alpha+\alpha')\cdot\frac{1-\cos(\kappa)}{2}+\frac{1+\cos(\kappa)}{2}\right).
\end{align*}
\end{lemma}
Note that the function $\mathcal{L}(x,y)$ is defined for all $x,y\in\mathbb{R}$.

\begin{proof}
By the spherical cosine rule applied to the triangle with vertices $x,e_{n+2},x'$ we have
\begin{align*}
d(x,x')&=\arccos(\cos(\alpha)\cos(\alpha')+\sin(\alpha)\sin(\alpha')\cos(\kappa))\\
&=\arccos(\cos(\alpha+\alpha')+(1+\cos(\kappa))\sin(\alpha)\sin(\alpha'))\\
&\geq
\arccos\left(\cos(\alpha+\alpha')+\sin\left(\frac{\alpha+\alpha'}{2}\right)^2(1+\cos(\kappa))\right).
\end{align*}
The last inequality above follows from $\arccos$ being decreasing and the fact that $\sin\left(\frac{\alpha+\alpha'}{2}\right)^2\geq \sin(\alpha)\sin(\alpha')$, which in turn follows from Jensen's inequality applied to the function $x\mapsto\ln(\sin(x))$, which is concave in the interval $(0,\pi)$. The second equality in \Cref{PropLBdx} follows from the identity $\sin^2\left(\frac{x}{2}\right)=\frac{1-\cos(x)}{2}$.
\end{proof}

\begin{remark}\label{Lis1Lichi}
The function $\mathcal{L}(\alpha+\alpha',\kappa)$ is, by the spherical cosine rule, equal to the distance between two points of $\mathbb{S}^{n+1}$ at distance $x/2$ of $e_{n+2}$ whose geodesics to $e_{n+2}$ form an angle of $y$. In particular, for fixed $\kappa$ (resp. $\alpha+\alpha'$), $\mathcal{L}(\alpha+\alpha',\kappa)$ is $1$-Lipschitz with respect to $\kappa$ (resp. $\alpha+\alpha'$). Moreover, for fixed $\kappa$, $\mathcal{L}$ is increasing with respect to $\alpha+\alpha'$. 
\end{remark}

\subsubsection{Bounds for $\alpha+\alpha'$ and $\kappa$}

We need some bounds on $\alpha+\alpha'$ and $\kappa$ for some of our following arguments to work.

\begin{lemma}[Upper bound for $\kappa$]\label{k<1.4} We have $\kappa\leq1.4$.
\end{lemma}

\begin{proof}
We consider the following upper bounds for $d(F(x),F(x'))$ obtained from the triangle inequality:
\begin{multline*}
d(F(x),F(x'))\leq d(\sigma(x),\sigma(x'))+d(F(x),\sigma(x))+d(F(x'),\sigma(x'))\\\leq\mathcal{U}_6(\alpha+\alpha',\kappa):=\kappa+(\pi-\zeta_3)\cdot\left(\frac{\pi}{2}-\alpha\right)+(\pi-\zeta_3)\cdot\left(\frac{\pi}{2}-\alpha'\right)=\kappa+(\pi-\zeta_3)\cdot\left(\pi-(\alpha+\alpha')\right).
\end{multline*}
\begin{multline*}
d(F(x),F(x'))\leq d(p,p')+d(p,F(x))+d(p',F(x'))\\\leq\mathcal{U}_7(\alpha+\alpha',\kappa):=\zeta_3+(\pi-\zeta_3)\left((\alpha+\alpha')-\pi+2\right).
\end{multline*}
If $d(F(x),F(x'))>d(x,x')+\zeta_3$, that is, $d(F(x),F(x'))-d(x,x')-\zeta_3>0$, then we have $\min(\mathcal{U}_6,\mathcal{U}_7)(\alpha+\alpha',\kappa)-\mathcal{L}(\alpha+\alpha',\kappa)-\zeta_3>0$. 
However, we check numerically (as explained in \Cref{CompAssistIneq}) that this inequality cannot happen if $\kappa\in[1.4,\pi]$ and $\alpha+\alpha'\in[\pi-2,2]$. To do it, we take a grid of points in $R=[1.4,\pi]\times[\pi-2,2]$ with coordinates $(x,y)=(\alpha+\alpha',\kappa)$ spaced by $d=10^{-4}$, and we compute the value of $\mathcal{I}_1(x,y):=\min(\mathcal{U}_6,\mathcal{U}_7)(x,y)-\mathcal{L}(x,y)-\zeta_3>0$. We have included the graph of $\mathcal{I}_1$ in \Cref{Figurek14} below.

The program \texttt{k14\_ineq.py} in \cite{Git} finds that the maximum value is $\sim-0.031>-0.03$. So to conclude, it will be enough to check that if two points $(x,y)$ and $(x',y')$ in $R$ satisfy $|x-x'|,|y-y'|\leq\frac{10^{-4}}{2}$, then $|\mathcal{I}_1(x,y)-\mathcal{I}_1(x',y')|<0.03$. The functions $\mathcal{U}_6$ and $\mathcal{U}_7$ are $10$-Lipschitz (this follows easily from their expressions), and $\mathcal{L}$ is $1$-Lipshitz as noted in \Cref{Lis1Lichi}, so the result follows.
\end{proof}

\begin{figure}[ht]
    \centering
    \includegraphics[width=0.7\linewidth]{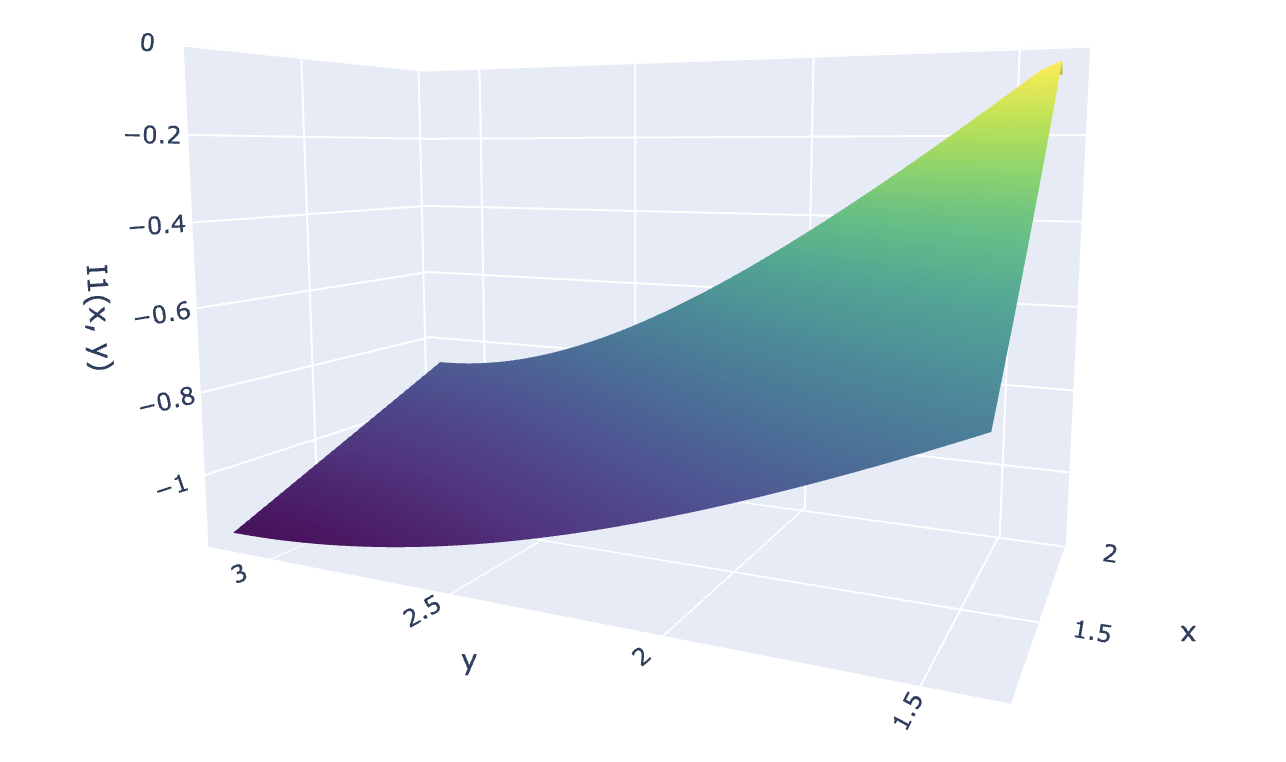}
    \caption{Graph of $\mathcal{I}_1(x,y)$.}
    \label{Figurek14}
\end{figure}

\begin{lemma}\label{defC(k)}
We have
\begin{equation}
\label{EqdefC(k)}
\cos(d(p,F(x')))\geq C(\kappa):=\frac
{-\sqrt{1-2\cos(\zeta_3-\kappa)+16\cos^2(\zeta_3-\kappa)}}
{\sqrt{15}}.
\end{equation}
\end{lemma}
See \Cref{GraphCk} for the graph of $C(\kappa)$.

\begin{figure}[ht]
    \centering
    \includegraphics[width=0.6\linewidth]{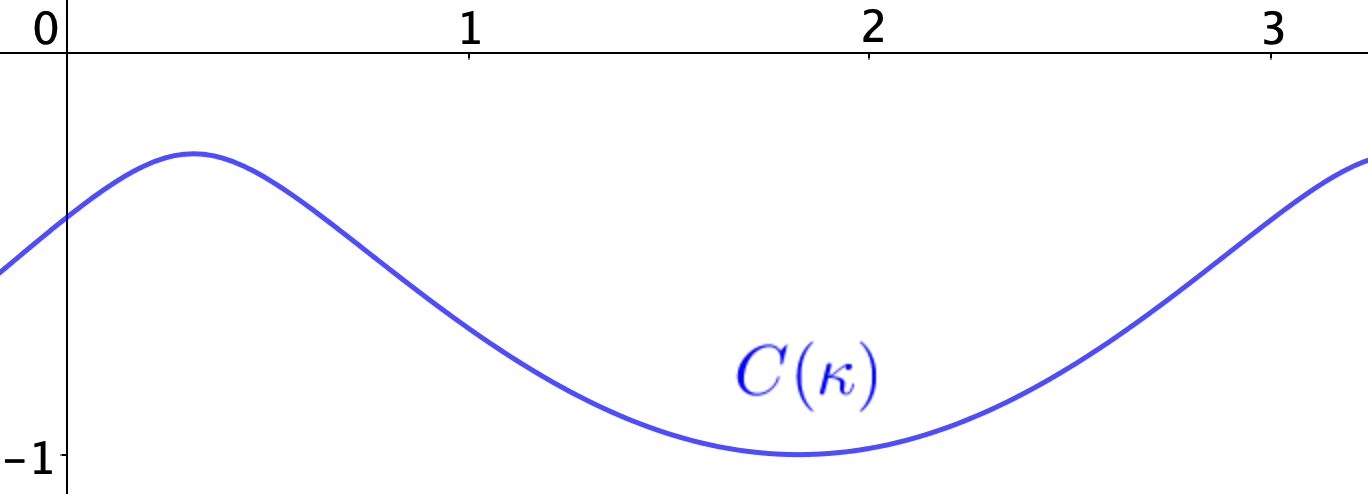}
    \caption{Graph of $C(\kappa)$.}
    \label{GraphCk}
\end{figure}

\begin{proof}
We prove that $\arccos(C(\kappa))$ is an upper bound for the distance from $p$ to any points on the segment $[\sigma(x'),p']$. Note that $d(p,p')=\zeta_3$, $d(p',\sigma(x'))\leq\pi-\zeta_3$ and $d(p,\sigma(x'))\leq\pi-\zeta_3+\kappa<\pi$ (the last inequality being by \Cref{k<1.4}). By \Cref{DistPtSg3} we can assume $d(p,p')=\zeta_3$, $d(p',\sigma(x'))=\pi-\zeta_3$ and $d(p,\sigma(x'))=\pi-\zeta_3+\kappa$. So by \Cref{DistSegPt}, the from $p$ to all points of the segment $[\sigma(x'),p']$ is bounded above by the function $C(\kappa)$ from \Cref{EqdefC(k)}.
\end{proof}

\begin{remark}\label{C(k)is1Lichi}
We will need the fact that $C(\kappa)$, or equivalently, $C_1(x):=-C(\zeta_3-x)=\sqrt{\frac{1-2\cos(x)+16\cos^2(x)}{15}}$, is $1$-Lipschitz. It is enough to check that $|C_1'(x)|\leq1$ for all $x$. Indeed, we have $$C_1'(x)=\frac{1}{\sqrt{15}}\cdot\frac{-\sin(x)+16\cos(x)\sin(x)}{\sqrt{1-2\cos(x)+16\cos^2(x)}}=\frac{1}{\sqrt{15}}\cdot\frac{\sin(x)(16\cos(x)-1)}{\sqrt{1-2\cos(x)+16\cos^2(x)}},$$ so if $y=\cos(x)$ then $C_1'(x)^2=\frac{(1-y^2)(16y-1)^2}{15(1-2y+16y^2)}$. One then readily checks that $(1-y^2)(16y-1)^2\leq15(1-2y+16y^2)$ for all $y\in\mathbb{R}$, concluding the proof.
\end{remark}

\begin{lemma}
We have $\kappa>0.3$.
\end{lemma}

\begin{proof}
Suppose for contradiction that $\pi-\zeta_3+\kappa<\arccos\left(\frac{-1}{16}\right)$(so $\kappa<0.315$ approximately). Then we will prove that $d(F(x),F(x'))\leq\zeta_3$, contradicting $d(F(x),F(x'))> d(x,x')+\zeta_3$. 
To prove it we first check that $d(p,F(x'))<\zeta_3$, or in general that $p$ is at distance $<\zeta_3$ of all points of the segment $[p',\sigma(x')]$.

Note that the triangle with vertices $p,p',\sigma(x')$ has sides $d(p,p')=\zeta_3$, $d(p',\sigma(x'))\leq\pi-\zeta_3$ due to \Cref{Viprops}\ref{Viprops986865}, and $d(p,\sigma(x'))\leq d(p,\sigma(x))+d(\sigma(x),\sigma(x'))\leq\pi-\zeta_3+\kappa$. Thus, in order to maximize distance from $p$ to points of the segment $[p',\sigma(x')]$, we may assume thanks to \Cref{DistPtSg3} that $d(p',\sigma(x'))=\pi-\zeta_3$ and $d(p,\sigma(x'))=\pi-\zeta_3+\kappa$. Applying the cosine rule then gives that 
$$\cos(\angle pp'\sigma(x'))=\frac{\cos(\pi-\zeta_3+\kappa)-\cos(\pi-\zeta_3)\cdot \cos(\zeta_3)}{\sin(\pi-\zeta_3)\cdot \sin(\zeta_3)}>0,$$
as $\pi-\zeta_3+\kappa<\arccos(-1/16)$. So
the angle $\angle pp'\sigma(x')$ is $<\frac{\pi}{2}$, thus the maximum distance from $p$ to points of the segment $[p',\sigma(x')]$ is $d(p,p')=\zeta_3$. 

Now, knowing that $d(p,F(x'))<\zeta_3$, $d(p,\sigma(x))\leq\pi-\zeta_3$ and $d(F(x'),\sigma(x))\leq \kappa+\pi-\zeta_3$, we repeat the exact same reasoning to obtain that the distance from $F(x')$ to any point of the segment $[p,\sigma(x)]$ (and in particular $F(x)$) is at most $\zeta_3$. Therefore $d(F(x),F(x'))\leq\zeta_3$, as we wanted.
\end{proof}

\begin{lemma}
[Lower bound for $\kappa$, $\mathcal{U}_3(\kappa)$]\label{k>0.7} We have $\kappa\geq0.7$.
\end{lemma}

\begin{proof}
We know that $\kappa>0.3$, so we assume for the sake of contradiction that $\kappa\in(0.3,0.7)$. Now, as $\alpha+\alpha'\geq\pi-2$, by \Cref{PropLBdx} we have
\begin{equation*}
d(x,x')\geq\mathcal{L}(\alpha+\alpha',\kappa)\geq\mathcal{L}(\pi-2,\kappa)=\arccos\left(\cos(2)\frac{\cos(\kappa)-1}{2}+\frac{1+\cos(\kappa)}{2}\right).
\end{equation*}

We now give an upper bound $\mathcal{U}_3(\kappa)$ for $d(F(x),F(x'))$, or more generally, for the distance from $F(x')$ to all points in the segment $[p,\sigma(x)]$; our reasoning will conclude by contradiction when we check that no $\kappa\in(0.3,0.7)$ satisfies $\mathcal{L}(\pi-2,\kappa)+\zeta_3\leq\mathcal{U}_3(\kappa)$. First we upper bound the lengths of the sides of the triangle $p\sigma(x)F(x')$.
Letting $C(\kappa)$ be given by \Cref{EqdefC(k)}, we have $d(p,F(x'))\leq \arccos(C(\kappa))$ by \Cref{defC(k)}; also $d(p,\sigma(x))\leq\pi-\zeta_3$ by \Cref{Viprops}\ref{Viprops986865} and finally,
\begin{equation*}
d(F(x'),\sigma(x))
\leq d(F(x'),\sigma(x'))+d(\sigma(x'),\sigma(x))
\leq 
\pi-\zeta_3+\kappa.
\end{equation*}
Thus, applying \Cref{DistPtSg3} we may assume $d(F(x'),\sigma(x))=\pi-\zeta_3+\kappa$, $d(p,F(x'))= \arccos(C(\kappa))$ and $d(p,\sigma(x))=\pi-\zeta_3$. And then applying \Cref{DistSegPt} we conclude that
\begin{equation*}
d(F(x),F(x'))
\leq\mathcal{U}_3(\kappa):=
\arccos\left(
\frac{-4}{\sqrt{15}}\sqrt{C(\kappa)^2+\cos(\pi-\zeta_3+\kappa)^2-\frac{1}{2}C(\kappa)\cos(\pi-\zeta_3+\kappa)}\right).
\end{equation*}
However, no value $\kappa\in(0.3,0.7)$ satisfies $\mathcal{L}(\pi-2,\kappa)+\zeta_3\leq\mathcal{U}_3(\kappa)$; we check this numerically using a $1$-dimensional version of the reasoning in \Cref{CompAssistIneq}. To do it, let $\mathcal{I}_2(\kappa)=\mathcal{U}_3(\kappa)-\mathcal{L}(\pi-2,\kappa)-\zeta_3$ (see the graph of $\mathcal{I}_2$ in \Cref{Figuretrgd5gfdvcx}). We first check using \texttt{Ineq2.py} from \cite{Git} that $\mathcal{I}_2(\kappa)<-0.025$ for all $x\in\{0.3,0.3+10^{-6},0.3+2\cdot10^{-6},\dots,0.7\}$. 
And then it remains to check that if $|x-x'|\leq\frac{10^{-6}}{2}$, then $|\mathcal{I}_2(x)-\mathcal{I}_2(x')|<0.025$. Firstly, the same reasoning at the end of the proof of \Cref{k<1.4} implies that, if $|x-x'|\leq\frac{10^{-6}}{2}$, then $|\mathcal{L}(\pi-2,x)-\mathcal{L}(\pi-2,x')|<0.005$ (using that $|\arccos(1)-\arccos(1-2\cdot10^{-6})|<0.005$). So it is enough to prove that, if $|x-x'|\leq\frac{10^{-6}}{2}$, then $|\mathcal{U}_3(x)-\mathcal{U}_3(x')|<0.02$. First note that the function $C(x)$ is $1$-Lipschitz (see \Cref{defC(k)}). 
Moreover, we have $\mathcal{U}_3(x)=\mathcal{N}\left(\sqrt{\mathcal{H}(x)}\right)$, where $\mathcal{N}(y)=\arccos\left(
\frac{-4\sqrt{y}}{\sqrt{15}}\right)$ and 
\begin{equation*}
\mathcal{H}(x):=C(x)^2+\cos(\pi-\zeta_3+x)^2-\frac{1}{2}C(x)\cos(\pi-\zeta_3+x).
\end{equation*}
Now, $x\mapsto C(x)^2$ is $2$-Lipschitz (as $|C(x)|,|C'(x)|\leq1$ for $x\in[0,\pi]$), $x\mapsto\cos(\pi-\zeta_3+x)^2$ is $2$-Lipschitz and $x\mapsto \frac{1}{2}C(x)\cos(\pi-\zeta_3+x)$ is $1$-Lipschitz, so $\mathcal{H}(x)$ is $5$-Lipschitz. Thus, if $|x-x'|<\frac{10^{-6}}{2}$, then $|\mathcal{H}(x)-\mathcal{H}(x')|<10^{-5}$. So we will conclude if we check that for all $y,y'$ with $|y-y'|<10^{-5}$ we have $|\mathcal{N}(y)-\mathcal{N}(y')|<0.02$. Using that the derivative $\mathcal{N}'(y)=\frac{2}{\sqrt{y(-16y+15)}}$ only gets big when $y$ approaches $0$ or $y=\frac{15}{16}$, we obtain that the biggest possible values of $|\mathcal{N}(y)-\mathcal{N}(y')|$ are $|\mathcal{N}(0)-\mathcal{N}(10^{-5})|<0.004$ and $|\mathcal{N}\left(\frac{15}{16}\right)-\mathcal{N}\left(\frac{15}{16}-10^{-5}\right)|<0.004$, so we are done.

\end{proof}
\begin{figure}[ht]
    \centering
    \includegraphics[width=0.6\linewidth]{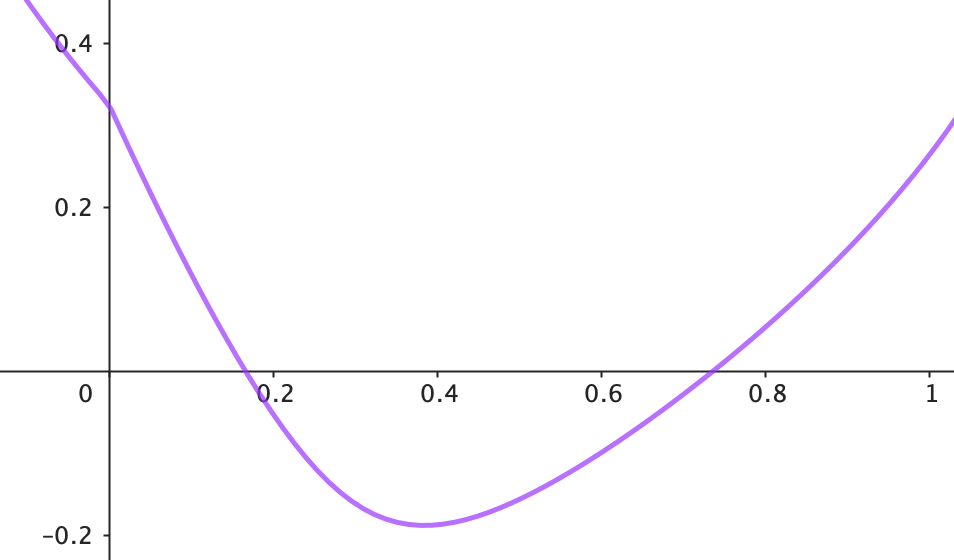}
    \caption{Graph of $\mathcal{U}_3(\kappa)-\mathcal{L}(\pi-2,\kappa)-\zeta_3$.}
    \label{Figuretrgd5gfdvcx}
\end{figure}

\begin{lemma}\label{Boundsxy}
We have $\kappa\in[0.7,1.4]$ and $\alpha+\alpha'\in\left[\pi-2,2\right]$.
\end{lemma}
\begin{proof}
All inequalities have been proved in \Cref{UBx,k>0.7,k<1.4} except $\alpha+\alpha'\geq\pi-2$, which follows from the assumption $\alpha,\alpha'\geq\frac{\pi}{2}-1$.
\end{proof}

\subsubsection{Upper bounds for $d(F(x),F(x'))$}

First we obtain an upper bound for $d(F(x),F(x'))$ which is useful when $\kappa$ is small. We  need the following lemma.

\begin{theorem}[Upper bound $\mathcal{U}_1$ for $d(F(x),F(x'))$]
\label{5ytrfgd4re}
Letting $\kappa_2:=\kappa+(\pi-\zeta_3)\left(\frac{\pi}{2}-\frac{\alpha+\alpha'}{2}\right)$, if $\pi-\zeta_3+\arccos(C(\kappa))+\kappa_2<2\pi$, then 
\begin{equation*}
d(F(x),F(x'))\leq\arccos\left(
\frac{-4}{\sqrt{15}}\sqrt{C(\kappa)^2+\cos(\kappa_2)^2-\frac{1}{2}C(\kappa)\cos(\kappa_2)}\right).
\end{equation*}
If on the other hand, $\pi-\zeta_3+\arccos(C(\kappa))+\kappa_2\geq2\pi$, then $d(F(x),F(x'))\leq\pi$.
\end{theorem}
We let $\mathcal{U}_1(\alpha+\alpha',\kappa)$ be the upper bound for $d(F(x),F(x'))$ given in \Cref{5ytrfgd4re} (given by different expressions depending on whether $\pi-\zeta_3+\arccos(C(\kappa))+\kappa_2<2\pi$ or not).

\begin{proof}
We assume without loss of generality $\alpha<\alpha'$. By the triangle inequality we have 
\begin{equation*}
d(F(x'),\sigma(x))\leq d(F(x'),\sigma(x'))+d(\sigma(x'),\sigma(x))\leq(\pi-\zeta_3)\left(\frac{\pi}{2}-\alpha'\right)+\kappa\leq\kappa_2.
\end{equation*}
So the triangle with vertices  $F(x'),\sigma(x),p$ has side lengths $d(p,F(x'))\leq\arccos(C(\kappa))$, $d(p,\sigma(x))\leq\pi-\zeta_3$ and $d(F(x'),\sigma(x))\leq \kappa_2$. If the sum of these three lengths is $<2\pi$ then by \Cref{DistPtSg3} and \Cref{DistSegPt} we conclude.
\end{proof}

\begin{theorem}[Upper bound $\mathcal{U}_2$ for $d(F(x),F(x'))$]\label{ThmUBdF2}
\[
d(F(x),F(x'))
\leq\mathcal{U}_2(\alpha+\alpha',\kappa):=
(\pi-\zeta_3)\left(\frac{\alpha+\alpha'}{2}-\frac{\pi}{2}+1\right)
+
G\left(\kappa,\frac{\alpha+\alpha'}{2}-\frac{\pi}{2}+1\right),
\]
where the function $G$ is defined as follows:
\[
G(\kappa,\beta'):=\arccos\left(\frac{-1}{4}\cos((\pi-\zeta_3)\beta')+\frac{\sqrt{15}}{4}\sin((\pi-\zeta_3)\beta')\frac{16\cos(\pi-\zeta_3+\kappa)+1}{15}\right).
\]
\end{theorem}
\begin{proof}
We assume $\alpha'\geq\alpha$. We will use the inequality 
\[d(F(x),F(x'))\leq d(F(x),p)+d(p,F(x'))\leq(\pi-\zeta_3)\left(\alpha-\frac{\pi}{2}+1\right)+d(p,F(x')).\]
In order to find an upper bound for $d(p,F(x'))$ we first consider a point $q_0\in\mathbb{S}^4$ such that $d(p',q_0)=\pi-\zeta_3$ and $d(p,q_0)=\pi-\zeta_3+\kappa$. We let $P_0:=(\alpha'-\frac{\pi}{2}+1)q_0\oplus\left(\frac{\pi}{2}-\alpha'\right)p'$. Note that if in the construction of $P_0$, we had instead asked $d(p',q_0)=d(p',\sigma(x'))\leq\pi-\zeta_3$ and $d(p,q_0)=d(p,\sigma(x'))\leq d(x,\sigma(x))+d(\sigma(x),\sigma(x'))\leq\pi-\zeta_3+\kappa$, then we would have $d(p,P_0)=d(p,F(x'))$, as $F(x')=(\alpha'-\frac{\pi}{2}+1)\sigma(x')\oplus\left(\frac{\pi}{2}-\alpha'\right)p'$. So we now prove that increasing the distances from $p,p'$ to $q_0$ also increases the distance $d(p,P_0)$.

\begin{claim}\label{tregfder4}
$d(p,F(x'))<d\left(p,P_0\right)$.
\end{claim}
\begin{proof}[Proof of \Cref{tregfder4}]
The points $x$ and $x'$ are fixed, so $\kappa,\alpha'$ will remain constant during this proof. Let $\beta'=\alpha'-\frac{\pi}{2}+1\in[0,1]$. 

For each $a\in[0,\pi-\zeta_3]$ and $b\in[0,\pi-\zeta_3+\kappa]$ with $a+b\geq\zeta_3=d(p,p')$ we consider some point $q=q(a,b)\in\mathbb{S}^4$ given by $d(p',q)=a$ and $d(p,q)=b$ (so $q$ could be $\sigma(x')$ for adequate $a,b$), and let $F(a,b)=\beta'q\oplus\left(1-\beta'\right)p'$. We will be done if we prove that the values $a_0,b_0$ which 
maximize $d(p,F(a,b))$ are $a_0=\pi-\zeta_3$, $b_0=\pi-\zeta_3+\kappa$, so $d(p,F(a_0,b_0))=d(p,P_0)$.

\begin{figure}[ht]
    \centering
    \includegraphics[width=0.45\linewidth]{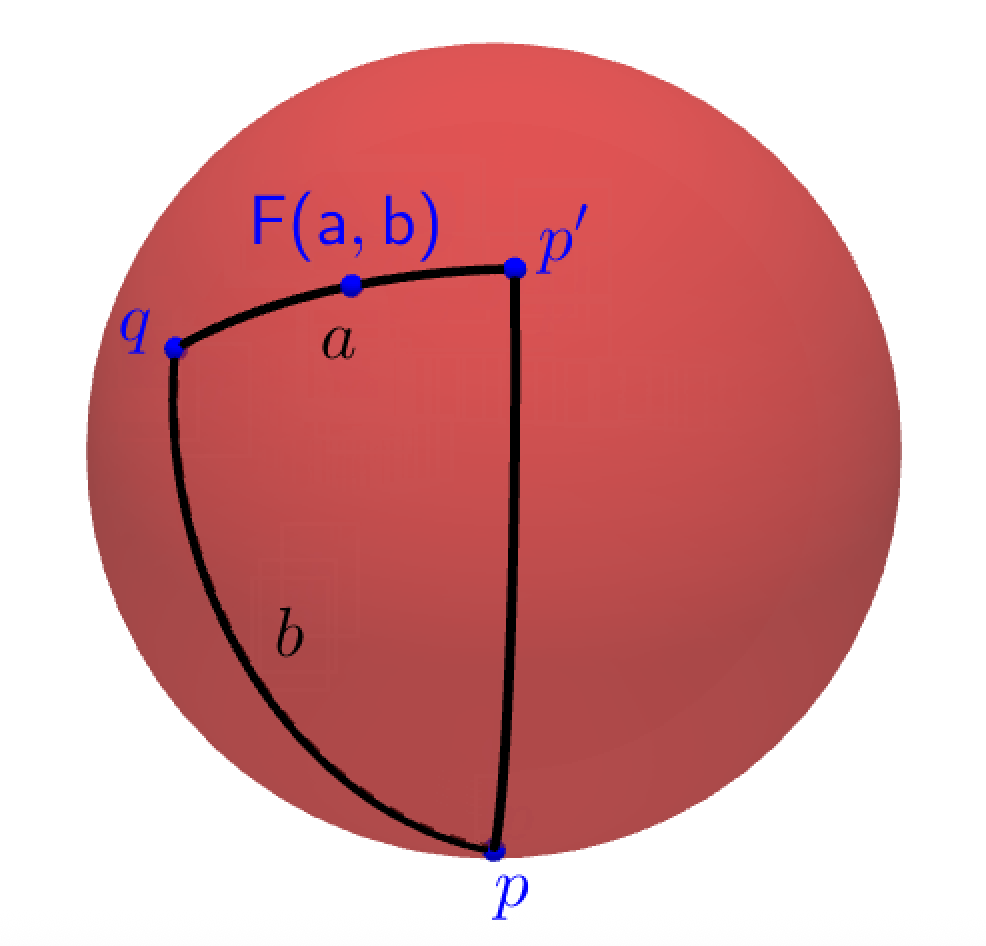}
    \caption{Construction of $q$ and $F(a,b)$ from $p,p',a$ and $b$, drawn in $\mathbb{S}^2$. We have $a+b\leq\pi$.}
    \label{Claim5-11Fig}
\end{figure}

Let us first prove $b_0=\pi-\zeta_3+\kappa$. We first note that for any fixed value of $a$, $d(p,F(a,b))$ is increasing with $b$ \footnote{If we fix $a$ and vary $b$, the point $F(a,b)$ forms a circles $c_F$ centered at $p'$ and with $p,-p$ both outside $c_F$, see \Cref{Claim5-11Fig} (where we set $p=(0,0,-1)\in\mathbb{R}^3$). The distance $d(p,F(a,b))$ increases with the height ($x_3$-coordinate) of $F(a,b)$ in $\mathbb{R}^3$, which in turn increases with the distance $d(p,q)=b$.}. Thus, if $\zeta_3+a_0>\pi-\zeta_3+\kappa$ (so that the triangle inequality in the triangle with vertices $p,p',q(a,b)$ allows $b=\pi-\zeta_3+\kappa$), we necessarily have $b_0=\pi-\zeta_3+\kappa$. But\footnote{In the following note that $a$ can take the value $\pi-2\zeta_3+\kappa$, because $\kappa>0.7$.}
\begin{equation*}
d(p,F(a_0,b_0))\geq d(p,F((\pi-\zeta_3+\kappa)-\zeta_3,\pi-\zeta_3+\kappa))=\pi-\zeta_3+\kappa.
\end{equation*}
 Thus, $a_0=d(p',q(a_0,b_0))\geq d(p',F(a_0,b_0))\geq d(p,F(a_0,b_0))-\zeta_3=\pi-2\zeta_3+\kappa$, concluding the proof that $b_0=\pi-\zeta_3+\kappa$.

So it will be enough to check that $d(p,F(a,b_0))$ is increasing with $a$. This can be proved using convexity arguments, but in this special case we can compute $d(p,F(a,b_0))$ explicitly: we first apply the cosine rule twice, obtaining
\begin{equation*}
\cos(\angle pp'F(a,b_0))
=
\cos(\angle pp'q(a,b_0))
=
\frac{\cos(b_0)-\cos(a)\cos(\zeta_3)}{\sin(a)\sin(\zeta_3)}.
\end{equation*}
\begin{align*}
\cos(d(p,F(a,b_0)))&=\cos(\zeta_3)\cos(a\beta')+
\sin(\zeta_3)\sin(a\beta')\frac{\cos(b_0)-\cos(a)\cos(\zeta_3)}{\sin(a)\sin(\zeta_3)}\\
&=
-\frac{\cos(a\beta')}{4}+
\frac{\sin(a\beta')}{\sin(a)}\left(\cos(b_0)+\frac{\cos(a)}{4}\right).
\end{align*}
So, computing the derivative with respect to $a$, we obtain 
\begin{equation*}
\frac{\partial}{\partial a}\cos(d(p,F(a,b_0)))
=
\frac{\beta'\sin(a\beta')}{4}+\frac{\sin(a\beta')}{\sin(a)}\cdot\frac{-\sin(a)}{4}+\frac{\partial}{\partial a}\left(\frac{\sin(a\beta')}{\sin(a)}\right)\left(\cos(b_0)+\frac{\cos(a)}{4}\right).
\end{equation*}
Using that $\beta'\in[0,1]$ and $\cos(b_0)<\frac{-1}{4}$, one readily checks that both the sum of the first two terms above and the third term above are negative for all $a\in(0,\pi)$. So $d(p,F(a,b_0))$ increases with $a$, as we wanted.
\end{proof}

Now, by the cosine rule (using $d(p,p')=\zeta_3$) we have that
\[
\cos(\angle(p,p',q_0))=\frac{\cos(\pi-\zeta_3+\kappa)+\frac{1}{16}}{\frac{15}{16}}=\frac{16\cos(\pi-\zeta_3+\kappa)+1}{15}.
\]

Thus, applying the cosine rule to the triangle $pp'P_0$ and letting $\beta'=\alpha'-\frac{\pi}{2}+1$, so that $d(p',P_0)=(\pi-\zeta_3)\beta'$, we obtain
\begin{align}\label{DefGkbeta}
d(p,F(x'))
&\leq d\left(p,P_0\right)=:G(\kappa,\beta')
=\\
&\arccos\left(\frac{-1}{4}\cos((\pi-\zeta_3)\beta')+\frac{\sqrt{15}}{4}\sin((\pi-\zeta_3)\beta')\frac{16\cos(\pi-\zeta_3+\kappa)+1}{15}\right).\label{4re90wopds9erwodsk}
\end{align}

Letting the RHS in \Cref{DefGkbeta} be $G(\kappa,\beta')=d\left(p,(\alpha-\frac{\pi}{2}+1)q+\left(\frac{\pi}{2}-\alpha\right)p'\right)$, our upper bound is $d(F(x),F(x'))\leq(\pi-\zeta_3)\beta+G(\kappa,\beta')$. 
But note that for fixed $\alpha+\alpha'$ with $\alpha'\geq\alpha$, this upper bound will be minimized for $\alpha=\alpha'=\frac{\alpha+\alpha'}{2}$ (this follows from the fact that, due to the definition of $P_0$, for any fixed $\kappa$ the function $G(\kappa,\beta')$ is $(\pi-\zeta_3)$-Lipschitz in $\beta'$), completing the proof of \Cref{ThmUBdF2}.
\end{proof}

\subsubsection*{Concluding the proof of the case $\alpha,\alpha'>\frac{\pi}{2}-1$.}

We have a lower bound $\mathcal{L}:[\pi-2,2]\times[0.7,1.4]\to\mathbb{R}$ for $d(x,x')$ and three upper bounds $\mathcal{U}_1,\mathcal{U}_2,\mathcal{U}_6:[\pi-2,2]\times[0.7,1.4]\to\mathbb{R}$ for $d(F(x),F(x'))$ in terms of $\alpha+\alpha'$ and $\kappa$ (see \Cref{Boundsxy} for the definition of $\mathcal{U}_6$). Letting 
$\mathcal{I}_3(x, y):=
\min(\mathcal{U}_1,\mathcal{U}_2,\mathcal{U}_6)(x,y)-\mathcal{L}(x,y)-\zeta_3$,
it will be enough to check that $\mathcal{I}_3(x, y)<0$ for all $(x,y)\in R:=[\pi-2,2]\times[0.7,1.4]$. 

As explained in \Cref{CompAssistIneq}, it is enough to check that $\mathcal{I}_3(x,y)<-0.08$ for all points $(x,y)$ in a grid in the rectangle $R$ with coordinates spaced by $d=10^{-5}$, which is done numerically in the file \texttt{Case1\_ineq.py} from \cite{Git}, and then check that if two points $(x,y)$ and $(x',y')$ in $R$ satisfy $|x-x'|,|y-y'|\leq\frac{10^{-5}}{2}$, then $|\mathcal{I}_3(x,y)-\mathcal{I}_3(x',y')|<0.08$\label{CheckUnifCont}. This follows from the facts that:
\begin{itemize}
    \item By \Cref{Lis1Lichi}, if $|x-x'|,|y-y'|\leq\frac{10^{-5}}{2}$, then $|\mathcal{L}(x,y)-\mathcal{L}(x',y')|<10^{-5}$. 

    \item If $|x-x'|,|y-y'|\leq\frac{10^{-5}}{2}$, then $|\mathcal{U}_1(x,y)-\mathcal{U}_1(x',y')|<0.07$. Similarly with $\mathcal{U}_2,
    \mathcal{U}_6$. This is easy to see for $\mathcal{U}_2,
    \mathcal{U}_6$ (in those cases we can change $0.07$ by $10^{-4}$), we focus on $\mathcal{U}_1$. To check it for $\mathcal{U}_1$ one can use that the function $\kappa\mapsto C(\kappa)$ is $1$-Lipschitz (see \Cref{C(k)is1Lichi}), that $C(\kappa)^2+\cos(\kappa_2)^2-\frac{1}{2}C(\kappa)\cos(\kappa_2)\geq3\frac{C(\kappa)^2}{4}\geq0.15$ for all $\kappa\in[0.7,1.5]$ and the uniform continuity constants of $x\mapsto\sqrt{x},x\mapsto\arccos(x)$.
\end{itemize}

\begin{figure}[ht]
    \centering
    \includegraphics[width=0.73\linewidth]{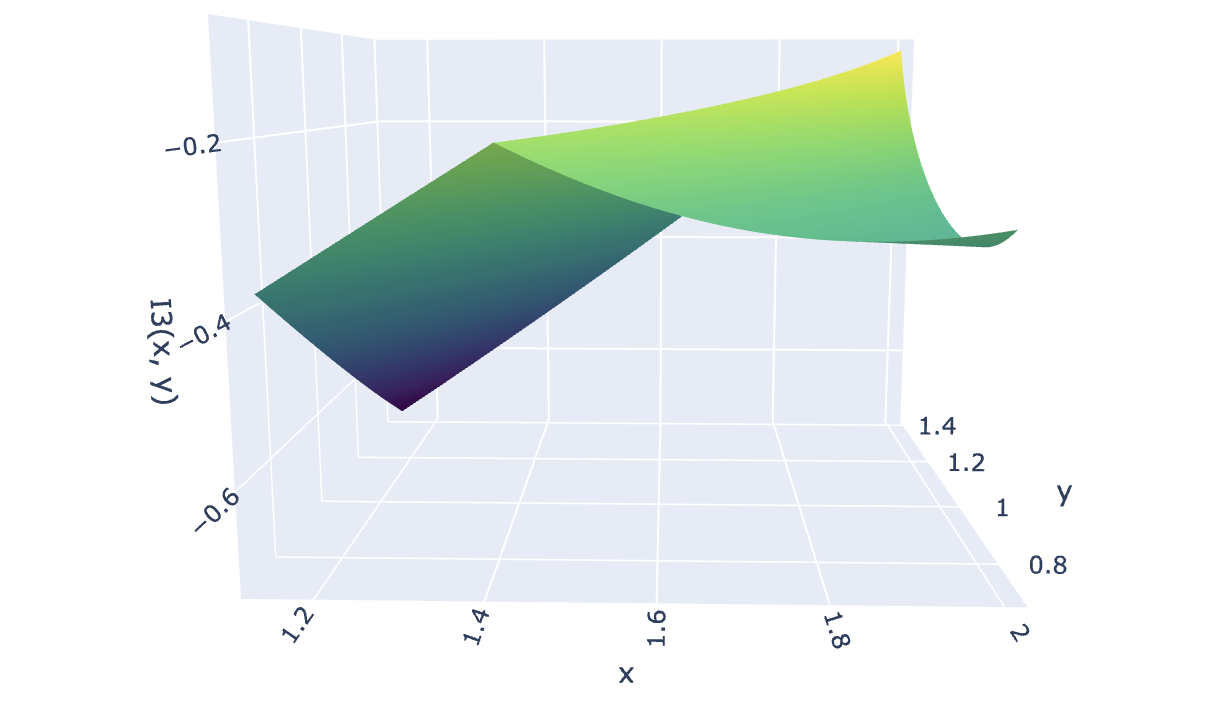}
    \caption{Graph of $\mathcal{I}_3(x,y)$.}
    \label{GraphI3}
\end{figure}

Finally, we consider the case where either $\alpha$ or $\alpha'$ are less or equal than $\frac{\pi}{2}-1$.

We will assume $\alpha'<\frac{\pi}{2}-1$, so $F(x')=p'$. By the reasonings of \Cref{k>0.7}, we may still assume $k>0.3$. We may assume $\alpha>\pi/2-1$, lest $d(F(x),F(x'))$ be $\zeta_3$. In this case we will also prove the inequality numerically using upper and lower bounds for $d(F(x),F(x'))$ and $d(x,x')$ respectively.

Note that, similarly to \Cref{ThmUBdF2}, we can deduce that, if $\beta=\alpha-\frac{\pi}{2}+1$, 
\begin{equation*}
d(F(x),F(x'))=d(F(x),p')
\leq G(\kappa,\beta).
\end{equation*} 
We also have several lower bounds for $d(x,x')$ in this case.
Due to \Cref{Lis1Lichi} we have $d(x,x')\geq\mathcal{L}(\alpha+\alpha',\kappa)\geq\mathcal{L}(\alpha,\kappa)$. 
We also know that $d(x,x')\geq\mathcal{L}_2(\alpha,\kappa)$, where $\mathcal{L}_2(\alpha,\kappa)=\alpha$ when $\kappa\geq\pi/2$ (as in that case, we can lower the value of $\alpha'$ to $0$, which reduces $d(x,x')$ and does not change $d(F(x),F(x'))$) and when $\kappa\leq\frac{\pi}{2}$, $\mathcal{L}_2(\alpha,\kappa)$ is just the distance from $x$ to the geodesic $e_{n+2}x'$, that is, 
\begin{equation*}
\mathcal{L}_2(\alpha,\kappa)=\arcsin(\sin(\alpha)\sin(\kappa)).
\end{equation*}

However, we can check using the reasonings from \Cref{CompAssistIneq} that for no values $\alpha\in\left[\frac{\pi}{2}-1,\frac{\pi}{2}\right]$ and $\kappa\in[0.3,\pi]$ can we have $G\left(\kappa,\alpha-\frac{\pi}{2}+1\right)>\max(\mathcal{L}(\alpha,\kappa),\mathcal{L}_2(\alpha,\kappa))+\zeta_3$. Equivalently, we need to check that, letting $$\mathcal{I}_4(x,y):=G(y,x-\pi/2+1)-\max(\mathcal{L}(x,y),\mathcal{L}_2(x,y))-\zeta_3,$$ we have $\mathcal{I}_4(x,y)\leq0$ for all $(x,y)\in\left[\frac{\pi}{2}-1,\frac{\pi}{2}\right]\times[0.3,\pi]$. See the graph of $\mathcal{I}_4(x,y)$ in \Cref{Graph_Case2}.

Firstly, we use \texttt{Case2\_ineq.py} in \cite{Git} to check that, for all points $(x,y)$ in a grid of points in the rectangle $\left[\frac{\pi}{2}-1,\frac{\pi}{2}\right]\times[0.3,\pi]$, with coordinates spaced by $10^{-5}$, we have $\mathcal{I}_4(x,y)\leq -0.16$. To conclude let $x,x'\in\left[\frac{\pi}{2}-1,\frac{\pi}{2}\right]$ and $y,y'\in[0.3,\pi]$ satisfy $|x-x'|<10^{-5}$ and $|y-y'|<10^{-5}$, we need to prove that $|\mathcal{I}_4(x,y)-\mathcal{I}_4(x',y')|\leq0.16$. Note that $\mathcal{L}$ is $2$-Lipschitz (see \Cref{Lis1Lichi}), and $\mathcal{L}_2$ is $2$-Lipschitz: as $\mathcal{L}_2$ is continuous in a convex domain, this can be checked separately for $\kappa\geq\pi/2$, where $\mathcal{L}_2(\alpha,\kappa)$ is $1$-Lipschitz, and for $\kappa\leq\pi/2$, where $\mathcal{L}_2$ is $1$-Lipschitz in $\alpha$ (for fixed $\kappa$) and in $\kappa$ (for fixed $\alpha$) because the function $x\mapsto\arcsin(a\cdot\sin(x))$ is $1$-Lipschitz for all $a\in[0,1]$, as it has derivative $\frac{a\cos(x)}{\sqrt{1-a^2\sin^2(x)}}$.

So it will be enough to check that, $|G(x,y)-G(x',y')|\leq0.15$. But letting $G(x,y)=\arccos(\mathcal{H}(x,y))$, where $\mathcal{H}$ is the big expression from \Cref{4re90wopds9erwodsk}, the facts that $\cos(x)$ and $\sin(x)$ are $1$-Lipschitz imply that $\mathcal{H}$ is $10$-Lipschitz, so $|\mathcal{H}(x,y)-\mathcal{H}(x',y')|\leq10^{-4}$. And for all $z,z'\in[-1,1]$ with $|z-z'|\leq 10^{-4}$ we have $$\left|\arccos(z)-\arccos(z')\right|<\left|\arccos(1)-\arccos(1-10^{-4})\right|<0.02,$$ so $|G(x,y)-G(x',y')|\leq0.02$ and we are done.

\begin{figure}[ht]
    \centering
    \includegraphics[width=0.7\linewidth]{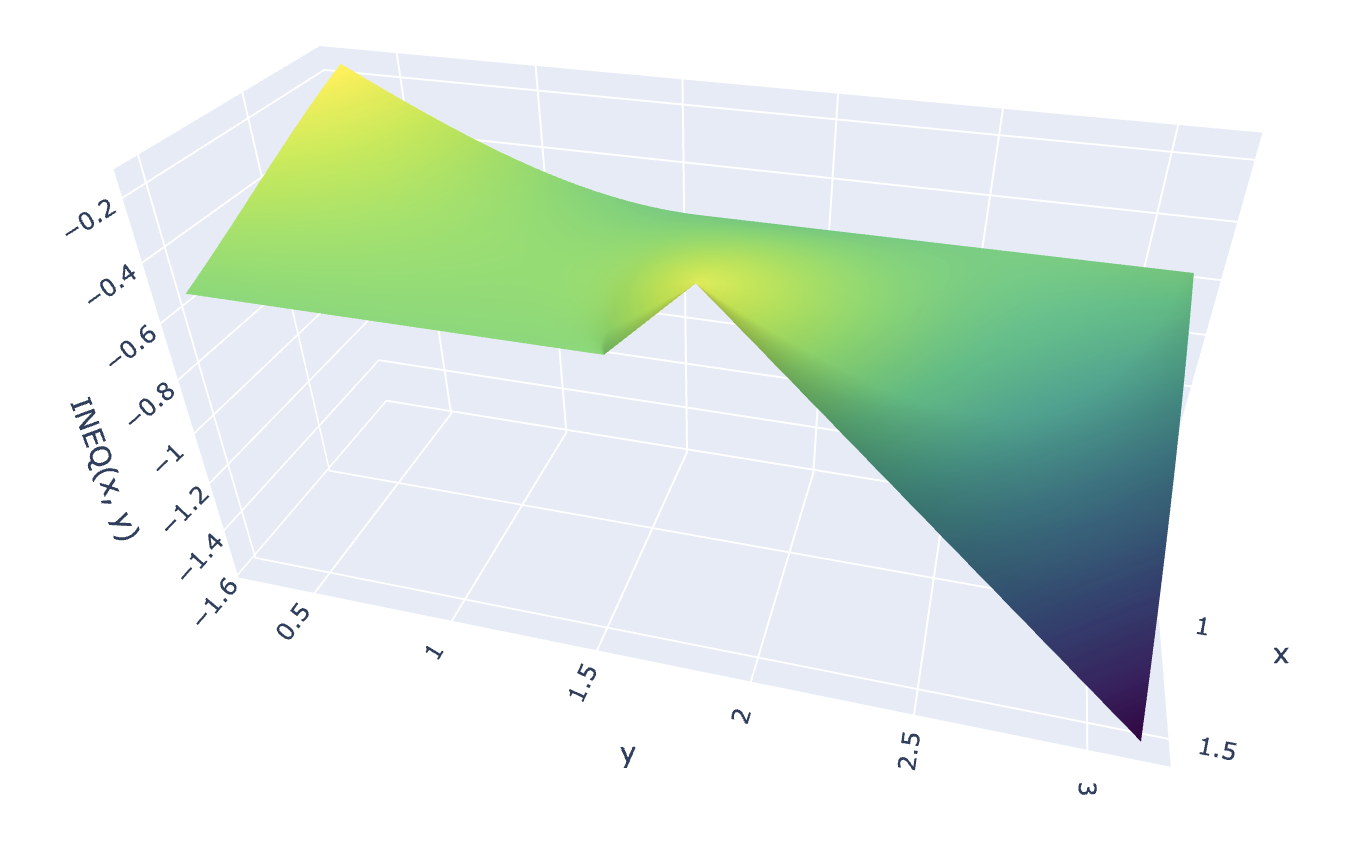}
    \caption{Graph of $G\left(y,x-\frac{\pi}{2}+1\right)-\max(\mathcal{L}(x,y),\mathcal{L}_2(x,y))-\zeta_3$}
    \label{Graph_Case2}
\end{figure}

\appendix

\msection{Spherical geometry lemmas used in \Cref{SecDS3S4}}
\label{SecAppendix}

We denote spherical distance by $d$ (instead of $d_{\mathbb{S}^n}$ for some $n$) in this section.

\begin{lemma}\label{54tret5regdf} 
Let $T$ be a spherical triangle with vertices $A,B,C$ and opposite side lengths $a,b,c$ respectively. If $b,c\leq\frac{\pi}{2}$ and the angle $\alpha$ at $A$ is $\leq\frac{\pi}{2}$, then $a\leq\frac{\pi}{2}$.
\end{lemma}

\begin{proof}
By the spherical cosine rule, 
$
\cos(a)=\cos(b)\cos(c)+\sin(b)\sin(c)\cos(\alpha)\geq0.
$\qedhere
\end{proof}

\begin{lemma}\label{ytrgftgfcb}
Let $T$ be a spherical triangle with vertices $A,B,C$ and opposite side lengths $a,b,c$ respectively. If $b,c\leq\frac{\pi}{2}$ and $a\geq\frac{\pi}{2}$, then we have $\alpha\geq a\geq\frac{\pi}{2}$, where $\alpha$ is the angle at $A$. If, moreover, $B',C'$ are points in the sides $AB$ and $BC$ respectively and $a'=d(B',C')$, then $a'\leq a$.
\end{lemma}

\begin{proof}
For the first part, by \Cref{54tret5regdf} we have $\alpha\geq\frac{\pi}{2}$, so by the spherical cosine rule, 
\begin{equation*}
\cos(a)=\cos(b)\cos(c)+\sin(b)\sin(c)\cos(\alpha)\geq\cos(\alpha).
\end{equation*}
For the second part, note that, in the formula above, $\cos(a)$ decreases when band grow.
\end{proof}

\begin{lemma}\label{jensen>pi/2}
Let $p,p'$ be two points in $\mathbb{S}^m$ ($m\geq2$) at distance $a$, and consider the hemispheres $H'=\{x\in\mathbb{S}^m;d(p',x)<d(p,x)\},
H=\{x\in\mathbb{S}^m;d(p,x)<d(p',x)\}$. Then for any point $x\not\in H'$ and any $\lambda\in[0,1]$, \[d(\lambda p\oplus(1-\lambda)x,H')\geq \lambda\frac{a}{2}.\]
\end{lemma}

\begin{proof}
We may assume $m=2$, as $p,p'$ and $x$ are contained in some isometric copy of $\mathbb{S}^2$ contained in $\mathbb{S}^m$. We may also assume $H=\{x\in\mathbb{S}^2\subseteq\mathbb{R}^3;x_{3}>0\}$ and $H'=\{x\in\mathbb{S}^m;x_{3}<0\}$, so $d(x,H')=\arcsin(x_3)$ for any $x\in\mathbb{S}^2$. Let $\gamma(t)$ be the geodesic joining $p$ and $x$; we will next prove that the function $t\mapsto d(\gamma(t),H')$ is concave in the intervals where $\gamma(t)\not\in H'$. We will then be done by Jensen's inequality, as $d(p,H')=\frac{a}{2}$.

Now, after reparameterizing $\gamma$ if necessary, the $x_3$-coordinate of our geodesic $\gamma(t)$ is $k\sin(t)$, for some $k\in[0,1]$. So $d(\gamma(t),H')=\arcsin(k\sin(t))$. This is a concave function in $[0,\pi]$, as its second derivative is 
\[
\frac{d^2}{dt^2}\arcsin(k\sin(t))
=
\frac{k(k^2-1)\sin(t)}{\left(1-k^2\sin^2(t)\right)^{3/2}}.
\]
When $k=1$, the second derivative does not exist at $t=\frac{\pi}{2}$, but in that case we have $d(O,\gamma(t))=\frac{\pi}{2}+\arcsin(\sin(t))$, which is again concave in $[0,\pi]$.
\end{proof}

\begin{lemma}\label{Axconvex}
If we let $\rho=\pi-\zeta_3=\arccos\left(\frac{1}{4}\right)$, the function $A(t)=\arccos(\cos(t)\cdot\cos(\rho t))$ is concave in the interval $(0,1)$.\footnote{Thanks to River Li for helping with the proof that $A(t)$ is concave in this MathStackExchange answer \href{https://math.stackexchange.com/a/4911668/807670}{(link)}.}
\end{lemma}

\begin{proof}
Let $c=\cos(t),s=\sin(t),c_1=\cos(\rho t),s_1=\sin(\rho t)$ during this proof (note that $c,c_1,s,s_1\in(0,1)$ for $t\in(0,1)$). 
The second derivative of $A(t)$ is given by the following formula
\begin{equation}
A''(t)=\frac{-1}{(1-c^2c_1^2)^{3/2}}(2\rho ss_1-cc_1\rho^2s^2-cc_1s_1^2).
\end{equation}

The formula was computed using the following code in \texttt{Python}/Sympy:

\begin{lstlisting}[language=Python]
import sympy as sp
t = sp.symbols('t')
f = sp.acos(sp.cos(t) * sp.cos(sp.acos(sp.Rational(1,4)) * t))
f2 = sp.diff(f, t, 2)
sp.simplify(f2)
\end{lstlisting}

So we just need to prove that $\forall t\in(0,1)$, $2\rho ss_1-cc_1\rho^2s^2-cc_1s_1^2>0$. This follows from the following two inequalities:
\begin{itemize}
    \item $\rho ss_1>cc_1\rho^2s^2$; this is equivalent to $s_1>cc_1\rho s$, which is true because for all $t\in(0,1)$, $\tan(\rho t)>\rho t\geq\rho\sin(t)\geq\rho\sin(t)\cos(t)$.
    \item $\rho ss_1>cc_1s_1^2$; in fact we have $\rho ss_1>s_1^2$, as $\rho\sin(t)>\sin(\rho t)$ for all $t\in(0,1)$.\qedhere
\end{itemize}
\end{proof}

\begin{lemma}[Upper bound for maximum distances from point to segment] \label{DistSegPt}
Let $u,v,w$ be the vertices of a spherical triangle in $\mathbb{S}^n$ with the sides opposite to $u,v,w$ having lengths $x_1,x_2,x_3$ respectively (so $x_1+x_2+x_3\leq2\pi$). Assume $v\not\in\{w,-w\}$. Then the maximum distance $\alpha$ between $u$ and any point of the geodesic passing through $v,w$ satisfies
\begin{equation}
\cos(\alpha)
=\frac{-\sqrt{\cos(x_2)^2+\cos(x_3)^2-2\cos(x_1)\cos(x_2)\cos(x_3)}}{\sin(x_1)}.
\end{equation}
\end{lemma}

\begin{figure}[ht]
    \centering
    \includegraphics[width=0.4\linewidth]{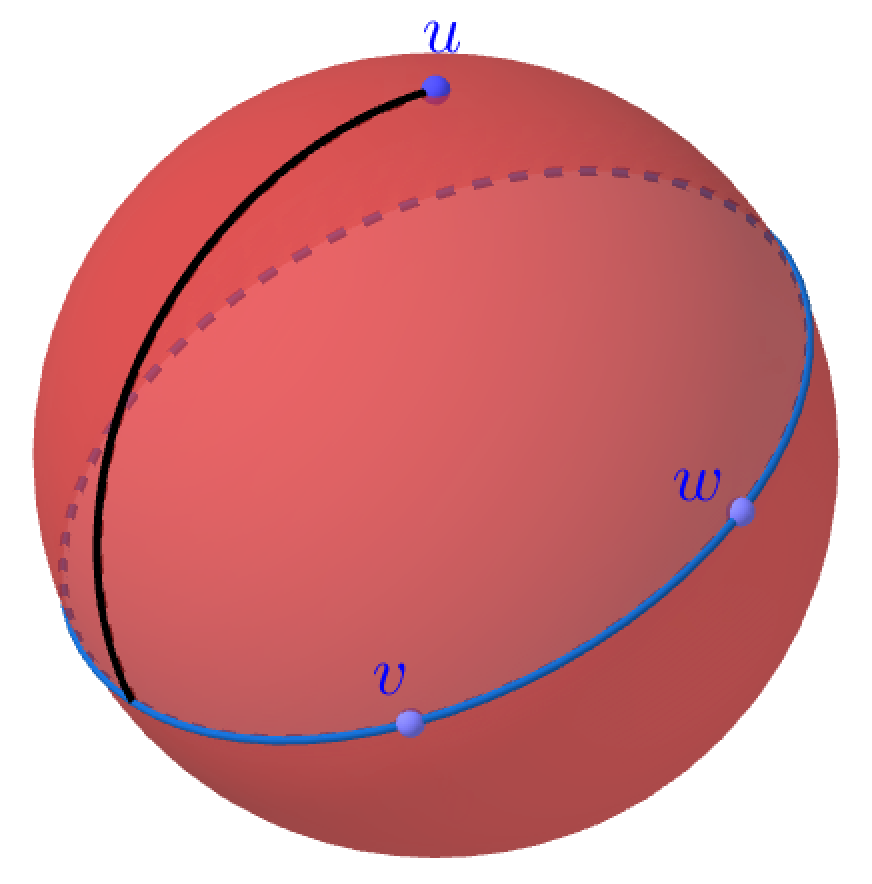}
    \caption{\Cref{DistSegPt} for $n=2$; the maximal distance from $u$ to the geodesic $vw$ is shown in black.}
\end{figure}

\begin{proof}
We may assume $n=2$, as $u,v,w$ are contained in an isometric copy of $\mathbb{S}^2$ inside $\mathbb{S}^n$.
Note that $\alpha\in[\frac{\pi}{2},\pi]$ is $\pi$ minus the angle between $u$ and the plane containing the origin, $v$ and $w$. Letting $a=\langle v,w\rangle=\cos(x_1),b=\langle u,w\rangle=\cos(x_2)$ and $c=\langle u,v\rangle=\cos(x_3)$, we have
\begin{align*}
\pi-\alpha=\arcsin\left(\frac{|\det(u,v,w)|}{\sqrt{1-a^2}}\right)
&=
\arcsin
\frac
{\sqrt{\det
\begin{pmatrix}
1&c&b\\c&1&a\\b&a&1
\end{pmatrix}
}}
{\sqrt{1-a^2}}\\
&=\arcsin\left({\frac{\sqrt{1+2abc-a^2-b^2-c^2}}{\sqrt{1-a^2}}}\right).
\end{align*}
So
\[
\cos(\alpha)
=
\frac{-\sqrt{b^2+c^2-2abc}}{\sqrt{1-a^2}}
=\frac{-\sqrt{\cos(x_2)^2+\cos(x_3)^2-2\cos(x_1)\cos(x_2)\cos(x_3)}}{\sin(x_1)}.\qedhere
\]
\end{proof}

\begin{lemma}\label{DistPtSgtFixedDistances}
Fix $p\in\mathbb{S}^2$ and $a\leq b$ numbers in $[0,\pi]$. For any  points $q,r\in\mathbb{S}^2$ with $r\not\in\{-q,q\}$, $d(p,q)=a$ and $d(p,r)=b$ let $t=d(q,r)$. Then the function $f(t)$ which gives the maximal distance from $p$ to any points in the geodesic segment $qr$ is well defined and increasing. More concretely,
\begin{itemize}
    \item If $a+b\leq\pi$, then $f:[b-a,b+a]\to\mathbb{R}$ is given by $f(t)=b$.
    \item If $a+b>\pi$, then $f:[b-a,2\pi-a-b]\to\mathbb{R}$ is given by $f(t)=b$ for all $t<\arccos\left(\frac{\cos(a)}{\cos(b)}\right)$. For $t\geq\arccos\left(\frac{\cos(a)}{\cos(b)}\right)$ $f$ is increasing, with 
\begin{equation}
\cos(f(t))
=\frac{-\sqrt{\cos(a)^2+\cos(b)^2-2\cos(t)\cos(a)\cos(b)}}{\sin(t)}.
\end{equation}
\end{itemize}
\end{lemma}

\begin{proof} 
Let $C_a,C_b\subseteq\mathbb{S}^2$ be the circles of points of $\mathbb{S}^2$ at distance exactly $a,b$ of $p$ respectively; during this proof, we consider $q,r$ as variable points taking values in $C_a,C_b$ respectively. Note that if $a+b\neq\pi$, the points $q,r$ are not antipodal, so the segment $qr$ is well defined. By the spherical cosine rule, the function $f(t)$ is well defined, i.e. it only depends on the distance between $p,q$. 

If $b<\pi/2$ then the closed ball centered at $p$ of radius $b$ is convex, so it contains the segment $qr$. If $b\geq\pi/2$ and $a+b<\pi$, then consider a closed hemisphere $H$ centered at some point at distance $b-\pi/2$ of $p$. Then $H$ contains some point $r$ at distance $b$ of $p$, and any segment between $r$ and points at distance $a$ of $p$ is entirely contained in $H$, proving that $f(t)\leq b$. The case $a+b=\pi$ with $p,q$ not antipodal follows from continuity.

Finally, assume that $a+b>\pi$, and fix some point $q$ at distance $a$ of $p$. The point $r$ can be any point in the circle $C=C_b$ of points at distance $b$ of $p$; as $r$ moves from the point in $C$ closest to $q$ to the the point in $C$ farthest from $q$, the quantity $t=d(q,r)$ increases from $b-a$ to $2\pi-a-b$.

\begin{figure}[ht]
    \centering
    \includegraphics[width=0.5\textwidth]{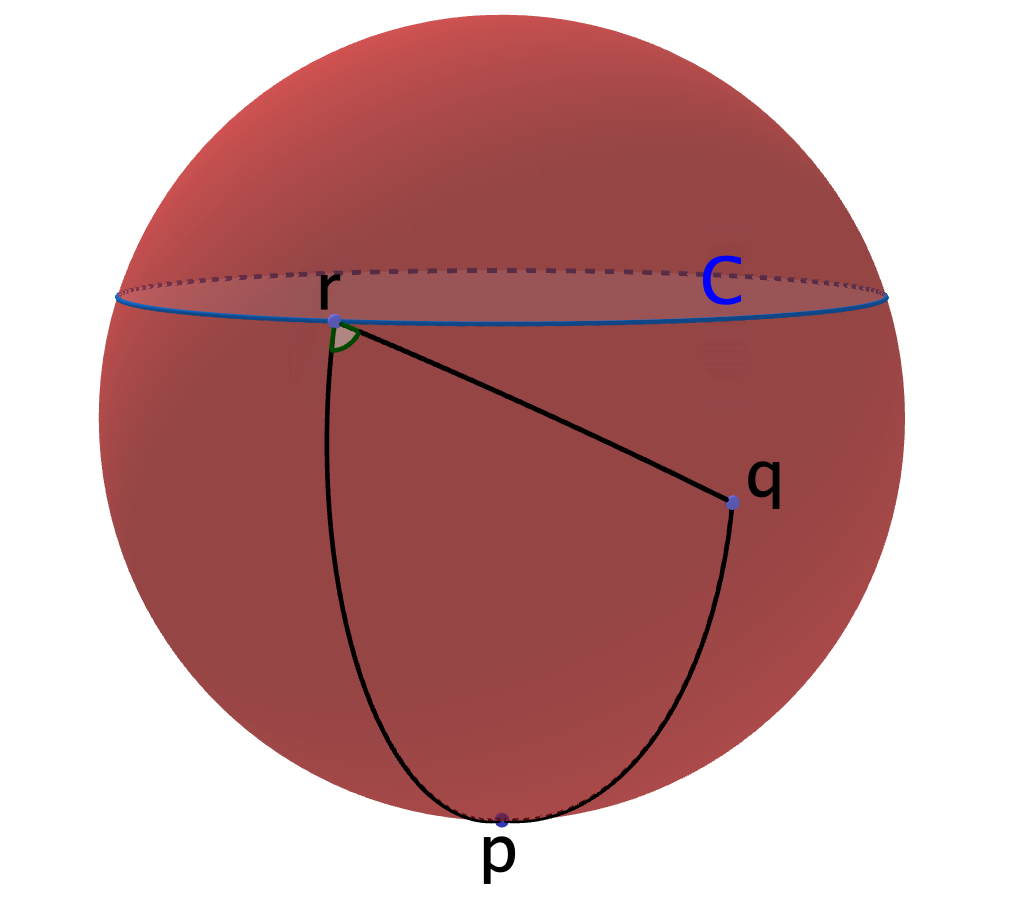}
    \hfill\includegraphics[width=0.45\textwidth]{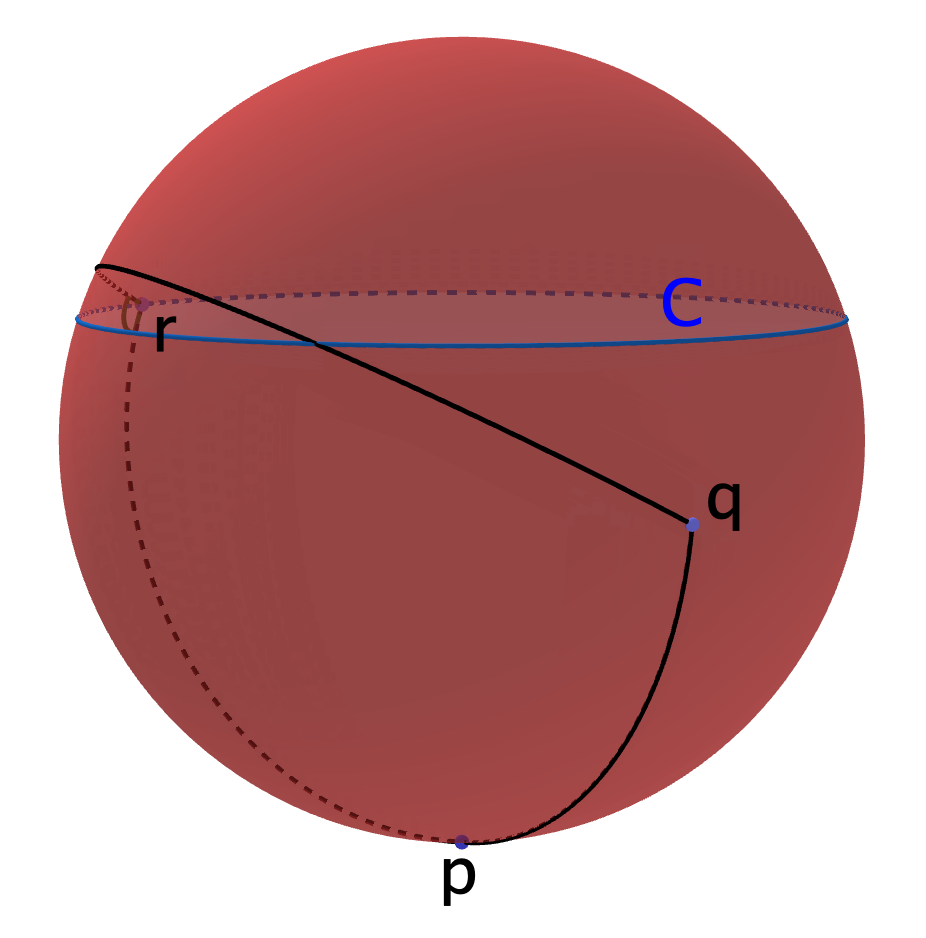}
    \caption{Points $p,q,r$ in the cases where $\angle prq$ is acute and obtuse.}
\end{figure}

If $\cos(t)\geq\frac{\cos(a)}{\cos(b)}$, then by the cosine rule the angle $\angle prq$ at $r$ is at most $\frac{\pi}{2}$, so $f(t)=d(p,r)=b$. If $\cos(t)<\frac{\cos(a)}{\cos(b)}$, then $\angle prq$ is obtuse, and the maximal distance from $p$ to points of the segment $qr$ is given by \Cref{DistSegPt}, that is, it is $\pi$ minus the angle between $\vec{0p}$ and the plane passing though $0,q,r$. Said angle, let's call it $\alpha(t)$ (so $\alpha(t)=\pi-f(t)$ for $t\geq\frac{\cos(a)}{\cos(b)}$), depends continuously on $t$ and is injective in the interval $\left[\frac{\cos(a)}{\cos(b)},2\pi-a-b\right]$, because each plane passing through $0,q$ only intersects the circumference $C$ in at most one point in such a way that the angle $\angle prq$ is obtuse. So as a continuous, injective function with $f(2\pi-a-b)\geq f\left(\frac{\cos(a)}{\cos(b)}\right)$, $\alpha(t)$ is decreasing in $\left[\frac{\cos(a)}{\cos(b)},2\pi-a-b\right]$.
\end{proof}

\begin{lemma}\label{DistPtSg3}
Let $a,b,c\in[0,\pi]$ satisfy $a\leq b+c,b\leq a+c,c\leq a+b$. We consider spherical triangles with vertices $x,y,z$ in $\mathbb{S}^2$, with $d(y,z)\leq a,d(x,z)\leq b,d(x,y)\leq c$. If $a+b+c\geq2\pi$, then the antipodal point $-x$ may lie in the geodesic segment $yz$. If $a+b+c<2\pi$, then the maximal distance from $x$ to any point of the segment $yz$ is reached when $d(y,z)= a,d(x,z)=b$ and $d(x,y)=c$.
\end{lemma}

\begin{proof}
First note that for any $a,b,c$ as above with $a+b+c<2\pi$ there are spherical triangles with sides $a,b,c$. If $a+b+c=2\pi$, then the union of the three sides of the triangle is a great circle (unless maybe $a,b$ or $c$ is $\pi$, in which case still $-x$ can lie inside the segment $yz$). So the only nontrivial case is $a+b+c<2\pi$. In particular, $a,b,c<\pi$ by the triangle inequality.

Let $f(\alpha,\beta,\gamma)$ be the maximal distance from $x$ to any point of the segment $yz$, where $d(y,z)= \alpha,d(x,z)=\beta$ and $d(x,y)=\gamma$. We want to prove that among all values $\alpha\leq a,\beta\leq b$ and $\gamma\leq c$, $f(\alpha,\beta,\gamma)$ will be maximized when $\alpha=a,\beta=b,\gamma=c$. This follows from the following facts:
\begin{itemize}
    \item For fixed values of $\beta,\gamma$, the function $f(\alpha,\beta,\gamma)$ increases with $\alpha$. This follows from \Cref{DistPtSgtFixedDistances}.
    \item For fixed values of $\alpha,\beta$, the function $f(\alpha,\beta,\gamma)$ increases with $\gamma$. When proving this we may assume $\alpha>0$ thanks to the previous fact, and also $\alpha\leq a<\pi$, so $z\not\in\{y,-y\}$ and we can apply \Cref{DistSegPt,DistPtSgtFixedDistances}.
    To see why $f(\alpha,\beta,\gamma)$ increases with $\gamma$, first consider the case when $\gamma\leq \pi-\beta$; in that case $f(\alpha,\beta,\gamma)=\max(\beta,\gamma)$ by \Cref{DistPtSgtFixedDistances}, so we are done. When $\gamma\geq\pi-\beta$, $f(\alpha,\beta,\gamma)$ will be (depending on whether the point in the segment $yz$ which is furthest from $x$ is one of the extremes or in the middle of the segment) either $\beta$, $\gamma$ or given by the formula from \Cref{DistSegPt}:
\begin{equation*}
\cos(f(\alpha,\beta,\gamma))=\frac{-\sqrt{\cos(\beta)^2+\cos(\gamma)^2-2\cos(\alpha)\cos(\beta)\cos(\gamma)}}{\sin(\alpha)}.
\end{equation*}
So, as $\beta$ and $\gamma$ are increasing functions of $\gamma$ ($\beta$ is just constant), it is enough to prove that the formula above is increasing. Or equivalently, that the function $g(\gamma)\mapsto\cos(\gamma)^2-2\cos(\alpha)\cos(\beta)\cos(\gamma)$ is increasing. But 
\begin{equation*}
g'(\gamma)=2\sin(\gamma)(\cos(\alpha)\cos(\beta)-\cos(\gamma))=-2\sin(\gamma)\sin(\alpha)\sin(\beta)\sin(C),
\end{equation*}
where $C$ is the angle opposite to the side of length $\gamma$ in a spherical triangle with sides of lengths $\alpha,\beta,\gamma$. It remains to prove that $C\geq\frac{\pi}{2}$. But we are assuming that the point $p$ furthest from $x$ in the geodesic $l$ passing through $y,z$, is contained in the segment $yz$. So we just have to check that, for any geodesic $\gamma$, if $p_0,p_1$ are points in $\gamma$ and $p_1$ is the point of $\gamma$ furthest from $x$, then the angle $\angle xp_0p_1$ is at least $\frac{\pi}{2}$. This follows from the fact that $t\to d(x,p(t))$ is increasing, where $p:[0,1]\to\mathbb{S}^2$ is the geodesic segment from $p_0$ to $p_1$, so $0\leq\left.\frac{d}{dt}\right|_{t=0}d(x,p(t))=-\cos(\angle xp_0p_1)$.
\item For fixed values of $\alpha,\gamma$, the function $f(\alpha,\beta,\gamma)$ increases with $\beta$. The proof is the same as in the previous case. 

\end{itemize} 
\end{proof}

\bibliographystyle{plain}
\bibliography{Bibliography}
\end{document}